\documentclass[format=acmsmall, review=false, screen=true]{acmart}

\usepackage{booktabs} % For formal tables

\usepackage[ruled]{algorithm2e} % For algorithms

\SetAlFnt{\small}
\SetAlCapFnt{\small}
\SetAlCapNameFnt{\small}
\SetAlCapHSkip{0pt}
\IncMargin{-\parindent}

\usepackage{mathtools}
\usepackage{multirow}
\usepackage{amsmath}
\usepackage{mathrsfs}

\usepackage[utf8]{inputenc}
\usepackage{amssymb}
\usepackage{nomencl}
\usepackage{etoolbox}

\allowdisplaybreaks

\begin{document}
% Title portion. Note the short title for running heads 
\title[A Dynamic Routing Framework for Shared Mobility Services]{A Dynamic Routing Framework for Shared Mobility Services}  

\thanks{This work was supported by the Ford-MIT Alliance.}

\thanks{Corresponding author: Yue Guan (\href{mailto:guany@mit.edu}{guany@mit.edu})}

\author{Yue Guan}
%\orcid{1234-5678-9012-3456}
\affiliation{%
  \institution{Massachusetts Institute of Technology}
  \streetaddress{77 Massachusetts Avenue}
  \city{Cambridge}
  \state{MA}
  \postcode{02139}
  \country{USA}}
\email{guany@mit.edu}
  \state{MA}

\author{Anuradha M. Annaswamy}
\affiliation{%
  \institution{Massachusetts Institute of Technology}
  \streetaddress{77 Massachusetts Avenue}
  \city{Cambridge}
  \state{MA}
  \postcode{02139}
  \country{USA}}
  \email{aanna@mit.edu}

\author{H. Eric Tseng} 
\affiliation{%
  \institution{Ford Motor Company}
  \streetaddress{1 American Road}
  \city{Dearborn} 
  \state{MI}
  \postcode{48126}
  \country{USA}}
\email{htseng@ford.com}

%\thanks{This work was supported by the Ford-MIT Alliance.}
%\thanks{$^{1}$A. M. Annaswamy, Y. Guan and T. Phan are with the Department of Mechanical Engineering, Massachusetts Institute of Technology, Cambridge, MA 02139, USA. {\tt\small \{aanna, guany, tkphan\}@mit.edu}}
%\thanks{$^{2}$H. E. Tseng, D. Yanakiev are with Research and Advanced Engineering (R\&A), and H. Zhou is with Global Data, Insights, and Analytics (GDI\&A), Ford Motor Company, Dearborn, MI 48121, USA. {\tt\small 
%		htseng@ford.com, zhouhowe@gmail.com, diana@ieee.org}}
%\thanks{Corresponding author: Yue Guan (guany@mit.edu)}

\begin{abstract}
Travel time in urban centers is a significant contributor to the quality of living of its citizens. Mobility on Demand (MoD) services such as Uber and Lyft have revolutionized the transportation infrastructure, enabling new solutions for passengers. Shared MoD services have shown that a continuum of solutions can be provided between the traditional private transport for an individual and the public mass transit based transport, by making use of the underlying cyber-physical substrate that provides advanced, distributed, and networked computational and communicational support. In this paper, we propose a novel shared mobility service using a dynamic framework. This framework generates a dynamic route for multi-passenger transport, optimized to reduce time costs for both the shuttle and the passengers and is designed using a new concept of a space window. This concept introduces a degree of freedom that helps reduce the cost of the system involved in designing the optimal route. A specific algorithm based on the Alternating Minimization approach is proposed. Its analytical properties are characterized. Detailed computational experiments are carried out to demonstrate the advantages of the proposed approach and are shown to result in an order of magnitude improvement in the computational efficiency with minimal optimality gap when compared to a standard Mixed Integer Quadratically Constrained Programming based algorithm.
\end{abstract}

%\end{CCSXML}

%\ccsdesc[500]{Computer systems organization~Embedded systems}
%\ccsdesc[300]{Computer systems organization~Redundancy}
%\ccsdesc{Computer systems organization~Robotics}
%\ccsdesc[100]{Networks~Network reliability}

%
% End generated code
%


\maketitle

% The default list of authors is too long for headers.
\renewcommand{\shortauthors}{Y. Guan et al.}

\makenomenclature

\setlength{\nomlabelwidth}{3cm}
 
%\mbox{}
 
\nomenclature{$n$}{numbers of passengers}
\nomenclature{$N$}{numbers of clusters}
%\nomenclature{$n, N$}{numbers of passengers and clusters, respectively}
%\nomenclature{$n, N$}{numbers of passengers and clusters, respectively}
%\nomenclature{$n, N$}{numbers of passengers and clusters, respectively}
\nomenclature{$i, j, k$}{indexes for ease of notation, $i, j$ for clusters, and $k$ for passengers}
\nomenclature{$Z_1, Z_2, Z_3, Z_4$}{sets defined for ease of notation on pages 5, 6, and 10}
\nomenclature{$P_k, D_k$}{requested pickup and drop-off locations of passenger $k$, respectively}
\nomenclature{$r_k^p, r_k^d$}{requested maximum walking distances before being picked up, and after being dropped off of passenger $k$, respectively}
\nomenclature{$T_k^r$}{time of request of passenger $k$}
\nomenclature{$d(\cdot, \cdot)$}{distance metric}
\nomenclature{$S_W(P_k, r_k^p)$}{space window of locations with distances from $P_k$ not exceeding $r_k^p$}
\nomenclature{$E_k^p, E_k^d$}{events of passenger $k$ being picked up and dropped off, respectively}
\nomenclature{$C_i$}{$i$th cluster}
\nomenclature{$A_i$}{area to visit of cluster $C_i$}
\nomenclature{$R_i$}{routing point of cluster $C_i$}
\nomenclature{$\mathcal{C}$}{objective function}
\nomenclature{$\gamma_1, \gamma_2$}{weights of travel time cost terms defined in the objective function}
\nomenclature{$\alpha_1, \alpha_2, \alpha_3^p, \alpha_3^d$}{weights of travel time cost terms defined in the objective function}
\nomenclature{$t_i$}{travel time of the shuttle on the $i$th trip segment}
\nomenclature{$WaitT_k$}{waiting time of passenger $k$}
\nomenclature{$RideT_k$}{riding time of passenger $k$}
\nomenclature{$Walk_k^p, Walk_k^d$}{walking times before being picked up, and after being dropped off of passenger $k$, respectively}
\nomenclature{$v_s, a, t_s, t_a$}{constants of the shuttle travel mode}
\nomenclature{$v_p$}{walking speed of the passengers}
\nomenclature{$p_k, d_k$}{orders of passenger $k$ in the pickup and drop-off queues, respectively}
\nomenclature{$MPS^p, MPS^d$}{upper bounds on the pickup and drop-off position shifts, respectively}
\nomenclature{$c$}{capacity of the shuttle}
\nomenclature{$O$}{origin of the shuttle}
\nomenclature{$S, S_f$}{sequence of the clusters, and the feasible domain of $S$, respectively}
\nomenclature{$R, R_f$}{routing points of the clusters, and the feasible domain of $R$, respectively}
\nomenclature{$x_{ij}$}{indicator that whether the shuttle travels from $A_i$ to $A_j$}
\nomenclature{$t_{ij}$}{travel time of the shuttle from $R_i$ to $R_j$}
\nomenclature{$T_i$}{departure time of the shuttle from $R_i$}
\nomenclature{$q_i$}{number of passengers on board after the shuttle departs from $R_i$}
\nomenclature{$f_i^p, f_i^d$}{numbers of finished pickups and drop-offs before the shuttle visits $A_i$, respectively}
\nomenclature{$l_i^p, l_i^d, l_i$}{numbers of pickups, drop-offs, and net loads of $C_i$}
\nomenclature{$w_k^p, w_k^d$}{walking times before being picked up, and after being dropped off of passenger $k$, respectively}
\nomenclature{$cl_k^p, cl_k^d$}{indexes of clusters containing $E_k^p$ and $E_k^d$, respectively}
\nomenclature{$R_{ij}^r, R_k^{s, p}, R_k^{s, d}, R_i^r$}{explaining decision variables for convex relaxation via SOC constraints defined in Equations (\ref{eqn:rotationmiqcp}), (\ref{eqn:shiftmiqcp}), (\ref{eqn:qcqprev}), and (\ref{eqn:qcqpsevd})}
\nomenclature{$T_{ij}, T_c$}{intermediate variables defined to linearize constraint (\ref{eqn:conmiqcp_6}) to (\ref{eqn:subtourmiqcp})}
\nomenclature{$Q_{ij}, Q_c$}{intermediate variables defined to linearize constraint (\ref{eqn:conmiqcp_9}) to (\ref{eqn:loadmiqcp})}
\nomenclature{$F_{ij}^p, F_{ij}^d, F_c$}{intermediate variables defined to linearize constraints (\ref{eqn:conmiqcp_10}) to (\ref{eqn:pickmiqcp}), and (\ref{eqn:conmiqcp_11}) to (\ref{eqn:dropmiqcp})}
\nomenclature{$S^h, R^h$}{$S$ and $R$ after the $h$th iteration defined in the AltMin algorithm, respectively}
\nomenclature{$g_i^1, g_i^2$}{numbers of passengers not picked up, and on board when the shuttle is travelling on the $i$th trip segment, respectively}
\nomenclature{$L, L'$}{indexes of the current and next possible clusters that the shuttle serves defined in the modified BHK algorithm, respectively}
\nomenclature{$t(L, L')$}{time taken for the shuttle to travel from $A_L$ to $A_{L'}$ defined in the modified BHK algorithm}
\nomenclature{$h$}{index of iteration defined in the AltMin algorithm}
\nomenclature{$h_{\text{max}}$}{upper bound on the number of iterations defined in the AltMin algorithm}
\nomenclature{$\bar{h}$}{number of iterations before termination defined in the AltMin algorithm}
\nomenclature{$h_0, h_0'$}{indexes of iterations between which the AltMin algorithm recurs, $h_0'<h_0$}
\nomenclature{$h^*$}{index of the iteration corresponding to the optimal cost defined in the AltMin algorithm} 
\nomenclature{$e_i, e_i'$}{indicators that whether $A_i$ has been visited at the current and possible next states defined in the modified BHK algorithm, respectively} 
\nomenclature{$\mathcal{E}$}{entropy of the state defined in the modified BHK algorithm} 
\nomenclature{$V$}{value function of the state defined in the modified BHK algorithm} 
\nomenclature{$M$}{multiplier of $t(L, L')$ contributing to $\mathcal{C}$ defined in the modified BHK algorithm} 
\nomenclature{$C^*, S^*, R^*$}{optimal cost, corresponding sequence, and routing points, respectively} 
\nomenclature{$\mathcal C_{\text{AltMin}}^\ast$}{optimal cost derived via the AltMin algorithm} 
\nomenclature{$\mathcal C_{\text{MIQCP}}^\ast$}{optimal cost derived via the MIQCP based approach} 

\printnomenclature
 
% makeindex -s nomencl.ist -o ACMpaper.nls ACMpaper.nlo

\section{Introduction}
\label{sec:intro}

Mobility of people and goods has been critical to urban life ever since cities emerged thousands of years ago. Over 1 billion vehicles travel on the roads today, and that number is projected to double by 2,050 \cite{ibm2015building}. New approaches and solutions are required to solve, or at least improve the quality of urban mobility, both on highways and on city streets. Until recently, available solutions for urban transportation have been clearly binary, with the first option represented by public transportation that provides low cost solutions and reduced carbon footprint per traveler at increased times of walking, waiting, and traveling, and the second represented by individual private automobiles that reduce individuals' times but with a significantly high cost and increased carbon footprint. With the ushering in Cyber-Physical Systems enabled by the development of smart mobile devices as well as affordable, accessible and powerful computing resources, a novel mode of urban mobility named Mobility on Demand (MoD) services has been revolutionizing the ground transportation infrastructure in urban centers \cite{ambrosino2004demand, mitchell2008mobility,chong2013autonomy}. Compelling statistics such as the idle rate of vehicles in the United States and United Kingdom being at 95\% \cite{reinventingparking} has enabled the emergence of several MoD companies worldwide, such as Uber, Lyft and Didi Chuxing, thereby enabling new solutions for passengers, with flexibility in travel time and cost. 

MoD solutions of late have started to provide several alternatives other than conventional transportation options, e.g., buses, subways and taxis, for passengers to move around. These correspond to different combinations of travel time, cost, and carbon footprint per traveler there by introduce different and attractive alternatives to passengers \cite{martin2011greenhouse, yu2017environmental}. A recent article \cite{alonso2017demand} illustrates clearly that the potential of shared mobility in drastically reducing the number of vehicles on the roads is very high. The article demonstrates that if ride sharing is implemented in New York City, 98\% of the taxi demands could be satisfied by 3,000 four-passenger cars, which is fewer than 25\% of the number of taxis currently operating in New York City with a waiting time of 2.7 minutes on average. Furthermore, over 95\% of the trips are feasible to be shared with other trips given a maximum 5-minute travel delay \cite{santi2014quantifying}. Similar empirical laws appear to be applicable in other cities as well, suggesting that the potential of shared mobility solutions is a global one \cite{tachet2017scaling}. This suggests that customized shared MoD solutions that involve multi-passenger transport can provide an ideal combination of convenience, flexibility, and affordability to passengers (Fig. \ref{fig:spectrum}) \cite{kampe, forddynamicshuttle, annaswamy2018transactive}. In this paper, we present one such shared MoD solution that uses a dynamic routing framework for operating a multi-passenger shuttle.
%However, there is still a gap in between the emerging MoD ride sharing services and conventional public transportations where the combination of flexibility, convenience and low tariff, large capacity can hardly be accommodated simultaneously (Fig. \ref{fig:spectrum}). In this paper, we present a novel dynamic routing framework for shared MoD solutions involving multi-passenger shuttles \cite{kampe, forddynamicshuttle, annaswamy2017transactive} that can provide an ideal combination of flexibility, convenience, and affordability to passengers.
\begin{figure}[h!]
	\centering
	\includegraphics[width=0.75\textwidth]{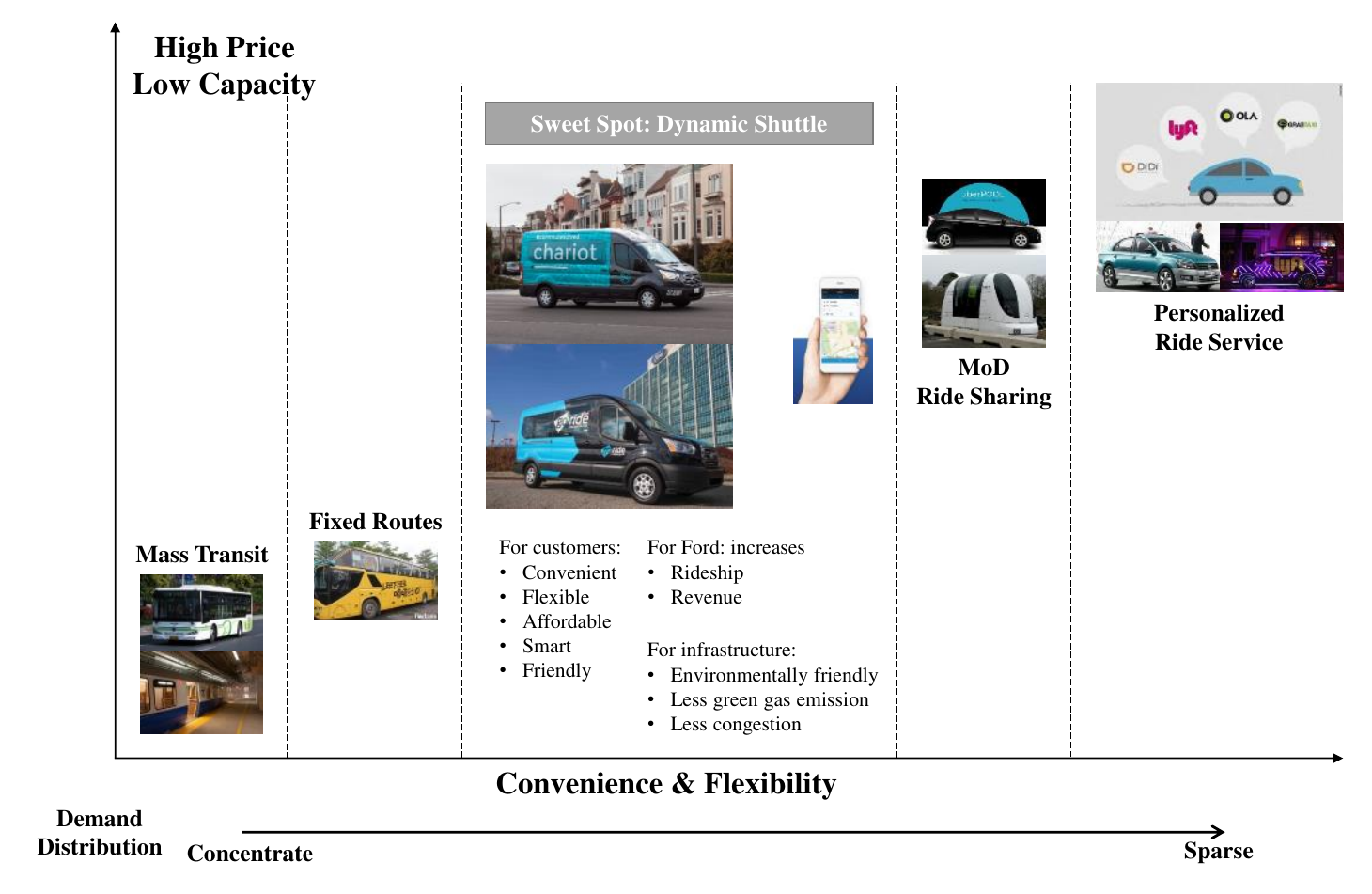}
	\caption{Spectrum of mobility services.}
	\label{fig:spectrum}
\end{figure}

The underlying goal in a shared mobility solution is the reduction of Price of Anarchy (PoA), which can be loosely defined as the cost ratio of all agents acting in a decentralized manner and that due to a completely centralized infrastructure \cite{youn2008price}. Our thesis is that our proposed shared MoD service can achieve this reduction by employing a dynamic framework, which responds to real time travel needs of passengers and incentivized so as to ensure their subscriptions to these routes. The specific dynamic framework that we propose is a dynamic shuttle service that responds to real time requests from passengers, determines a route that is customized to their requested pickup/drop-off locations, which optimizes a time cost that includes the total travel time of the shuttles as well as time costs incurred by the passengers. While this framework is also capable of incorporating dynamic prices for the shuttle rides \cite{annaswamy2018transactive}, we focus in this paper only on the dynamic routing of the shuttle. 

Dynamic routing is a main feature of ride sharing services. Ride sharing has been the subject of study in many papers including \cite{spieser2014toward, atasoy2015concept, santi2014quantifying, lioris2016optimised, alonso2017demand, tachet2017scaling}. Ref. \cite{spieser2014toward} is a macroscopic study of the Singapore mobility scene and estimates the financial benefits of ride sharing. References \cite{santi2014quantifying, alonso2017demand} focus on the simulation of actual implementations of ride sharing, with the latter providing a more elaborate case study. Ref. \cite{lioris2016optimised} discusses ride sharing using a case study of Paris and the underlying ride sharing strategies. Ref. \cite{zhang2015feeder} proposes a transit service to tackle the last-mile problem and validates its capability of reducing last-mile distances and travel times via large-scale simulation based upon real data. What makes our paper distinct and novel compared to these other works is the use of a new concept that we have introduced, termed as \textit{space window}. While the concept of dynamic routing is considered in many of these papers, none of them considers the added flexibility and efficiency in routing provided by space windows.
% Ref. \cite{zhang2013coride} designs a carpool service which can lower passenger fares and simultaneously increase driver profit.

The main idea of a space window is that the passengers are willing to walk for a certain distance, both before being picked up and after being dropped off in order to get the dynamic shuttle service. Passengers specify the maximum distances they are willing to walk and therefore the requested pickup/drop-off locations expand to pickup/drop-off regions denoted as space windows $S_W$, which are defined as the sets of all feasible locations that passengers are willing to walk to/from. Space windows are one of the key ingredients of the proposed dynamic shuttle service, which can therefore be viewed as a semi door-to-door mode (Fig. \ref{fig:semi}). Our thesis is that the introduction of $S_W$ provides greater flexibility and an important degree of freedom in reducing PoA. Moreover, as will be seen in the sections that follow, the introduction of $S_W$ allows a clustering of different pickup/drop-off events which is able to reduce the number of shuttle stops and in turn leads to a significant reduction in the time cost compared to a traditional MoD service. Further benefits are due to real street topologies and traffic congestion and the associated asymmetries between driving time and walking time, such as one-ways and pedestrian only lanes as a short walking distance for a passenger may result in significant improvements in the shuttle route. Such an increase in the quality of service has the potential to provide additional positive externalities on the customer riding experiences as well. An elaborate comparison of dynamic routing with space windows versus that without space windows has been conducted in \cite{annaswamy2018transactive} using real operational data, demonstrating an average improvement of 30\% reduction on travel time costs and 50\% reduction on the number of shuttle stops. However, space windows also introduces penalties in terms of extra walking times for passengers, leading to a trade-off in the shuttle service design. The maximum distance willing to walk is not only the protocol between a passenger and the shuttle server when negotiating a ride, but also a key tuning parameter to adjust the trade-off. One should also accommodate the fact that different customers may have different preferences in terms of how far they are willing to walk, or the same passenger may have different preferences depending on various factors including weather conditions, seasons, and times of day.
\begin{figure}[h!]
	\centering
	\includegraphics[width=0.75\textwidth]{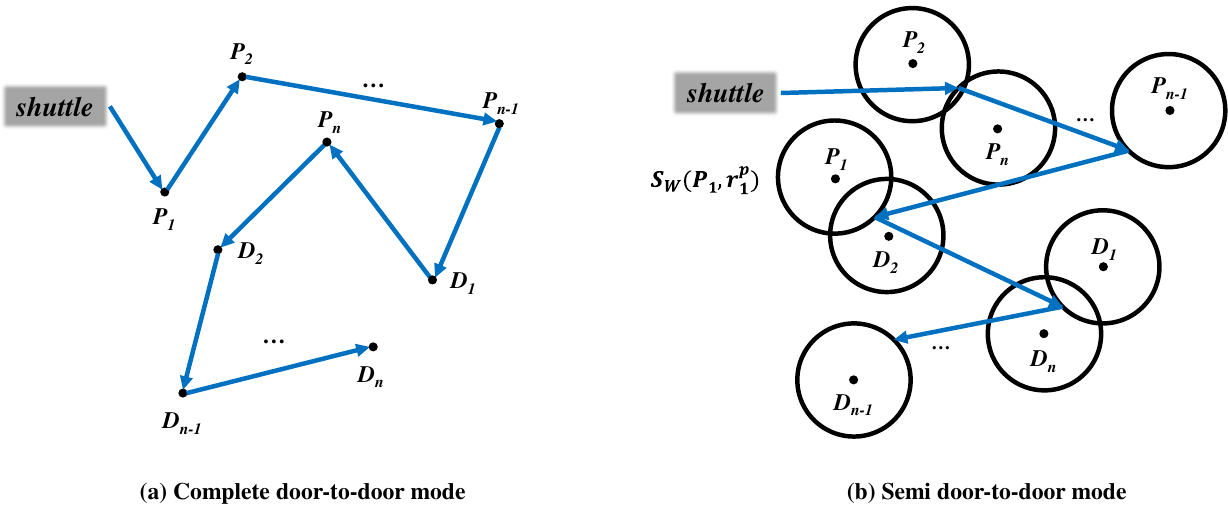}
	\caption{Complete door-to-door mode versus semi door-to-door mode.}
	\label{fig:semi}
\end{figure}

With the above concept of a space window, the overall shared mobility service is proposed to operate in four steps: 1) \textit{Request}: passengers request shuttle rides with specified pickup/drop-off locations, maximum distances willing to walk, and time windows of service if needed. 2) \textit{Offer}: the shuttle server distributes offers to passengers with ride details including pickup locations, walking distances, pickup times, drop-off locations, drop-off times, and prices. 3) \textit{Decide}: passengers decide whether to accept or decline the offers. 4) \textit{Operate}: the shuttle server sends out operational instructions to the shuttles according to the decisions from the passengers and trips start. 
%With the above concept of a space window, the overall shared mobility service is proposed to operate in four steps: 
%\begin{enumerate}
%	\item \textit{Request}: passengers request shuttle rides with specified pickup/drop-off locations, maximum distances willing to walk, and time windows of service if needed.
%	\item \textit{Offer}: the shuttle server distributes offers to passengers with ride details including pickup locations, walking distances, pickup times, drop-off locations, drop-off times, and prices.
%	\item \textit{Decide}: passengers decide whether to accept or decline the offers.
%	\item \textit{Operate}: the shuttle server sends out operational instructions to the shuttles according to the decisions from the passengers and trips start.
%\end{enumerate}
\begin{figure}[h!]
	\centering
	\includegraphics[width=0.75\textwidth]{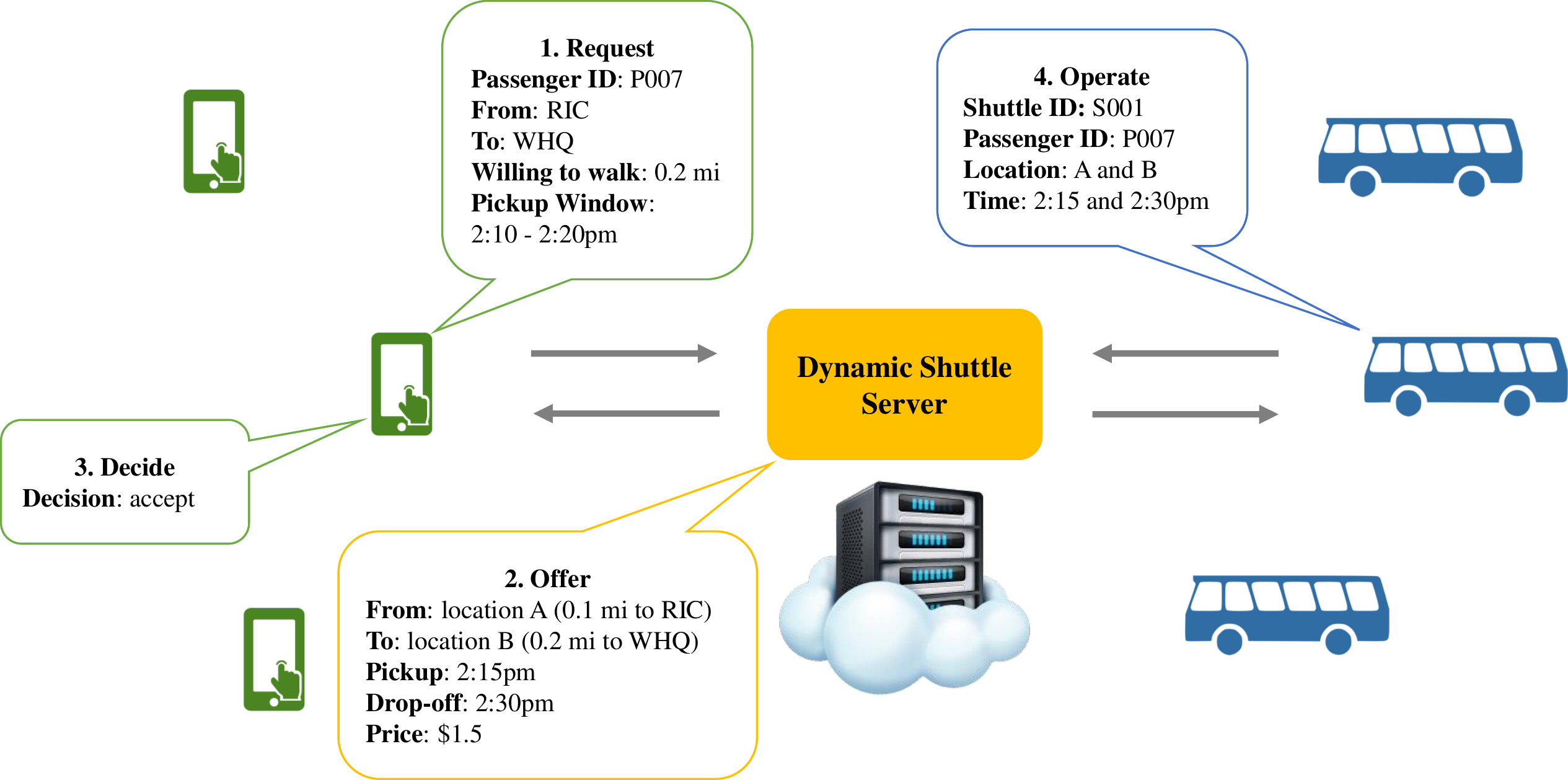}
	\caption{Procedures of the dynamic shuttle service: request, offer, decide, and operate.}
	\label{fig:procedure}
\end{figure}

The introduction of space windows however introduces significant challenges in designing an analytically tractable optimization framework. The overall problem of conventional dynamic routing with complete door-to-door mode, often referred to as a Dial-a-Ride Problem (DARP) has been studied extensively in the literatures. The underlying problem is essentially a Traveling Salesman Problem (TSP). The concept of space windows introduces extra complexity to the optimization framework, where dynamic routing with the therefore semi door-to-door mode is analogous to a TSP between regions as opposed to points. Though the overall problem can be formulated as a standard Mixed Integer Programming (MIP), it is extremely computationally burdensome. In this paper, we propose the use of an Alternating Minimization \cite{ortega2000iterative, auslender1976optimisation, luo1993error} based optimization approach in order to tackle the problem. In addition, we use this approach to carry out dynamic routing in response to real time requests from passengers, thereby facilitates an overall dynamic framework. This framework not only includes dynamic routing, which is the focus of this paper, but also dynamic pricing \cite{annaswamy2018transactive}.

In summary, the main contributions of this paper are the following: i) a novel dynamic routing framework for shared mobility services which utilizes multi-passenger shuttles, built upon a central concept termed as space window; ii) a constrained optimization formulation that shows how the space window concept can be utilized to carry out dynamic routing of shuttles, delineated by first carrying out static routing, for ease of exposition, and then the dynamic case; iii) an Alternating Minimization (AltMin) algorithm to solve the constrained optimization problem, benchmarked against a standard Mixed Integer Quadratically Constrained Programming (MIQCP) based approach; iv) quantifications of the advantages of the proposed AltMin algorithm using real operational data. As will be shown in the sections that follow, the proposed framework is able to reduce the number of shuttle stops and further save travel times for both the shuttle and the passengers. The AltMin algorithm is demonstrated to achieve an order of magnitude improvement in the computational efficiency and preserves optimality compared to the MIQCP based approach in various scenarios.

The overall organization of the paper is as follows. We formally state the problem of a shared mobility service with space window under a static routing setup for a single shuttle case in Section \ref{sec:form}, using a constrained optimization formulation. Before proceeding to detail the AltMin algorithm we propose to solve static routing, we describe the standard MIQCP formulation in Section \ref{sec:miqcp}, which will be used for the benchmark. In Section \ref{sec:altmin}, we describe the AltMin algorithm. Detailed computational experiments are carried out in Section \ref{sec:experiments}, using real operational data from \cite{fordchariot, kampe}, to evaluate the performances of the AltMin algorithm, and compare closely with the MIQCP based approach. Extensions to dynamic routing using its static counterpart as a subcomponent are discussed in Section \ref{sec:dynamic}. A summary and concluding remarks are provided in Section \ref{sec:conclusions}.

\section{Problem Formulation}
\label{sec:form}

In this section, we formulate the problem that we propose to address in the context of shared mobility services. The goal is to generate a route for a multi-passenger shuttle so as to optimize the underlying time costs. In this section we restrict our attention to static routing, where the assumption is that all ride requests have arrived prior to the departure of the shuttle, and once the trip starts, no further requests will be accepted until the current rides are completed. The overall goal of the shared mobility service is the generation of a static route using the four steps described in Fig. \ref{fig:procedure} in Section \ref{sec:intro}, which will be shown in this section, to be in the form of a constrained optimization problem. The static route that is proposed by the server is determined so as to optimize various time costs that are incurred both by the shuttle as well as by the passengers. In the former case, we include the shuttle travel time, and in the latter case, we include the walking times, waiting times, and ride times incurred by the passengers. Before we proceed to describe the algorithms, we first present a few preliminaries in Section \ref{subsec:pre}. In Section \ref{subsec:objfun}, we formally state the objective function, and in Section \ref{subsec:cons}, we delineate various constraints. The overall \textit{master problem} is then stated formally in Section \ref{subsec:master}.
%In this section, the problem of static routing with space windows is formulated via an optimization framework. Static routing corresponds to the case where the shuttle has collected all ride requests regarding the coming trip before departure. Once the trip starts, the shuttle will not accept new requests until the current trip is finished. Therefore the route will not update at anytime of the trip.
%A virtual time cost is added to represent service times at each shuttle stop. 

\subsection{Preliminaries}
\label{subsec:pre}
%\begin{definition}[Ride Request]
%A \textit{ride request} for the dynamic shuttle service a tuple $\{P_k, D_k, r_k^p, r_k^d\}$ composed of the customer specified pickup location $P_k$, drop-off location $D_k$, the maximum distance willing to walk before being picked up $r_k^p$ and after being dropped off $r_k^d$, respectively. $k \in S_1= \{ 1, 2, \dots, n \}$ is the order of the requests according to the time stamp when it is made, i.e., a smaller $k$ value corresponds to an earlier request. $n \in \mathbb Z_{> 0}$ is the number of requests, $P_k, D_k \in {\mathbb R}^2$, and $r_k^p, r_k^d \in \mathbb R_{\geq 0}$, $\forall k \in S_1$. Without loss of generality, assume there is only one passenger per request, and all the requests are made at the time stamp $T=0$. Denote $(P_k,r_k^p), (D_k,r_k^d) \in {\mathbb R}^2$ as the pickup and drop-off space window for the $k$th passenger, respectively. 
%\end{definition}

\begin{definition}[Ride Request]
	A \textit{ride request} of the dynamic shuttle service is defined as a tuple $\{P_k, D_k, r_k^p, r_k^d, T_k^r\}$, composed of $P_k$ the customer specified pickup location, $D_k$ drop-off location, $r_k^p$ and $r_k^d$ the maximum distances willing to walk before being picked up and after being dropped off, respectively, as well as $T_k^r$ the time of request. $P_k, D_k \in {\mathbb R}^2$ (e.g., longitude and latitude), $r_k^p, r_k^d \in \mathbb R_{\geq 0}$, and $T_k^r \in \mathbb R$. Let $Z_1= \{ 1, 2, \dots, n \}$, where $n \in \mathbb Z_{> 0}$ is the number of requests. $k \in Z_1$ denotes the order of the request according to the time stamp when it is made, i.e., a smaller $k$ value corresponds to an earlier request. Without loss of generality, we assume there is only one passenger per request, and all the requests are made at time $T_k^r=0, \forall k \in Z_1$. 
\end{definition}

\begin{definition}[Space Window]
	A pickup/drop-off \textit{space window} is defined as the set of all feasible pickup/drop-off locations whose distance from the requested pickup/drop-off location does not exceed the specified maximum walking distance. For each $k \in Z_1$, denote $S_W(P_k,r_k^p), S_W(D_k,r_k^d) \in {\mathbb R}^2$ as the pickup and drop-off space window for the $k$th passenger, respectively, such that 
	\begin{equation}
	\label{eqn:spacewindowpick}
	S_W(P_k,r_k^p)=\{x \, | \, x \in {\mathbb R}^2, d(x, P_k) \leq r_k^p\}
	\end{equation}
	\begin{equation}
	\label{eqn:spacewindowdrop}
	S_W(D_k,r_k^d)=\{x \, | \, x \in {\mathbb R}^2, d(x, D_k) \leq r_k^d\} 
	\end{equation}
	where $d$ is the distance metric on ${\mathbb R}^2$.
\end{definition}
Let $E_k^p$ and $E_k^d$ denote the events that correspond to the $k$th passenger being picked up and dropped off, respectively. If during the same trip, the space windows of different pickup/drop-off events share a common intersection, the corresponding events are assumed eligible to be grouped. In such a case, the shuttle is allowed to visit the intersection to take care of all the corresponding pickup/drop-off tasks simultaneously, instead of visiting individual space windows and serving the tasks one by one. 

\begin{definition}[Cluster]
A \textit{cluster} $C_i$ is a nonempty set of pickup and/or drop-off events if the corresponding space windows share a common intersection, where $i \in \mathbb Z_{>0}$ is the index of the cluster. 
\end{definition}
It should be noted that the existence of a shared intersection is necessary for clustering, but not sufficient. This will be discussed in Section \ref{subsec:master}.
% as it may not necessarily result in a reduction in cost of the system. This will be discussed later in this section.

\begin{definition}[Clustering Pattern]
	A \textit{clustering pattern} is a set of mutually exclusive clusters $\{C_1, C_2, \dots, C_N\}$, whose union constitutes the set of all pickup and drop-off events, where $N \in \mathbb Z_{> 0}$ denotes the number of clusters.
\end{definition}
Let $Z_2=\{1, 2, \dots, N\}$. By definition, $\bigcup_{i \in Z_2} C_i = \{E_k^p, E_k^d, \forall k \in Z_1 \}$. Also, $C_i \cap C_j = \emptyset, \forall i, j \in Z_2, i \neq j$. $N=2n$ corresponds to the case where no nontrivial clusters are formed, that is, each pickup/drop-off event itself is a cluster. Typically, $N<2n$ because the likelihood that different events are eligible to be clustered is fairly high for peak travel conditions in urban areas. It should be noted that some clusters may only contain pickups or drop-offs, and some may contain both.

\begin{definition}[Area]
	An \textit{area} $A_i, \forall i \in Z_2$ is defined as a subset of ${\mathbb R}^2$ that represents the feasible area to visit associated with the cluster $C_i$.
\end{definition}
It follows that $A_i$ is the intersection of the corresponding space windows of the events contained in $C_i$, and is given by
\begin{equation}
\label{eqn:areadef}
A_i=\{x \, | \, x \in S_W(P_{k_p}, r_{k_p}^p) \cap S_W(D_{k_d}, r_{k_d}^d), \forall E_{k_p}^p, E_{k_d}^d \in C_i \}
\end{equation}

\begin{definition}[Routing Point]
A \textit{routing point} $R_i \in A_i, \forall i \in Z_2$ is defined as the actual location that the shuttle visits in area $A_i$ to pick up and/or drop off the corresponding passengers.
\end{definition}

For the sake of consistency, throughout this manuscript, $Z_1$ and $Z_2$ will be used as the sets which contain the indexes of the passengers (or requests) and clusters, respectively. Moreover, subindex $i, j \in Z_2$ are associated with clusters, and subindex $k \in Z_1$ is associated with passengers.

The routing task of the dynamic shuttle service is essentially to design a route in order to serve a given set of ride requests, which minimizes the cost of the shuttle-passenger system and satisfies certain service constraints, which are defined as follows.

\subsection{Objective Function}
\label{subsec:objfun}

The \text{objective function} of the system is defined as follows
\begin{equation}
\label{eqn:objfundef}
\mathcal C=\gamma_1 \sum_{i=1}^{N} t_i + \gamma_2 \sum_{k=1}^{n} (\alpha_1 WaitT_k + \alpha_2 RideT_k + \alpha_3^p WalkT_k^p + \alpha_3^d WalkT_k^d)
\end{equation}
which represents a weighted sum of various time costs, both on the shuttle side and on the passenger side. In (\ref{eqn:objfundef}), $\forall i \in Z_2$, $t_i$ denotes the time taken by the shuttle to complete the $i$th trip segment, $\forall k \in Z_1$, $WaitT_k, RideT_k, WalkT_k^p$ and $WalkT_k^d$ denote the waiting time, riding time, walking time before being picked up and after being dropped off of passenger $k$, respectively, and $\gamma_1, \gamma_2, \alpha_1, \alpha_2, \alpha_3^p$ and $\alpha_3^d$ are nonnegative weights.

It should be noted that the objective function defined in (\ref{eqn:objfundef}) is a general representation which does not depend on any specific distance metric, i.e., versatile for various distance metrics as long as represented properly. Throughout this manuscript, the underlying space topology is assumed to be a two dimensional Euclidean space and as a result the corresponding space windows $S_w(P_k, r_k^p)/S_w(D_k, r_k^d)$ are circles centered at $P_k/D_k$ with $r_k^p/r_k^d$ as the radii, $\forall k \in Z_1$. A similar formulation can be extended to other metrics as well, such as the Manhattan distance and metrics that correspond to real street topologies.

As has been briefly discussed in Section \ref{sec:intro}, the reduction on $N$ has potential to facilitate the reduction on $\mathcal C$, especially when considering the actual travel profile of the shuttle. For each $i \in Z_2$, the term $t_i$ includes both the actual travel time on the $i$th trip segment and the service time at the stop. The travel time is calculated assuming a constant shuttle velocity $v_s$. However in reality, the shuttle accelerates and decelerates between stops. Therefore an extra amount of time $\frac{v_s}{2a}$, which is a reasonable approximation of the average additional time taken considering varying velocity profile, is added for correction. $a$ denotes the shuttle acceleration. Moreover, the shuttle stops at each station for a service time $t_s$ in order to serve the passengers. Hence $t_i$ is calculated as the sum of the travel time using constant velocity, and augmented with an additional time $t_a$ given by 
\begin{equation}
	\label{eqn:travelmodel}
	t_a=t_s+\frac{v_s}{2a}
\end{equation}
Therefore our formulation applies implicit penalty on $N$, as $t_a$ accumulates when the shuttle travels along the route. 

\subsection{Constraints}
	\label{subsec:cons}

The route of the dynamic shuttle service should satisfy the following \text{constraints}.

\begin{definition}[Legitimate Constraint]
	\label{def:legitimate}
	\textit{Legitimate constraints} specify that any passenger should be picked up before being dropped off, that being said, $\forall k \in Z_1$, $E_k^p$ should occur before $E_k^d$. Moreover, $E_k^P$ and $E_k^d$ cannot be in the same cluster as otherwise passenger $k$ would be willing to walk directly from the requested origin to the destination instead of taking a shuttle ride.
\end{definition}

\begin{definition}[Capacity Constraint]
	\label{def:capacity}
	\textit{Capacity constraints} specify that at anytime of the trip, the number of passengers on board should not exceed the capacity of the shuttle denoted as $c \in \mathbb Z_{>0}$. 
\end{definition}

It should be noted that the number of passengers within a trip could be larger than $c$, as long as they are not on board simultaneously at anytime during the trip.

\begin{definition}[Maximum Position Shift Constraint]
\label{def:mps}
A pickup/drop-off \textit{position shift} is defined as the absolute value of the difference between the actual occurrence of the pickup/drop-off event and that of the request event. \textit{Maximum Position Shift (MPS) constraints} specify that the position shifts should not exceed some given upper bounds, that being said
\begin{equation}
\label{eqn:mpsp}
|p_k-k| \leq MPS^p, \quad \forall k \in Z_1
\end{equation}
\begin{equation}
\label{eqn:mpsd}
|d_k-k| \leq MPS^d, \quad \forall k \in Z_1
\end{equation}
where $p_k$ and $d_k$ denote the orders of the $k$th passenger in the pickup and drop-off queue, respectively. $MPS^p$ and $MPS^d$ are positive integer upper bounds on the pickup and drop-off position shifts, respectively, which are provided by the dynamic shuttle server as part of the service protocol.
\end{definition} 
MPS constraints are priority constraints that serve as a guarantee of the dynamic shuttle service quality, such that passengers will not experience a long waiting time or travel delay. They are especially significant in the dynamic routing scenario in order to prevent the possibility of any passenger from indefinite deferment of pickup and drop-off by the algorithm.

From the above definitions, we note that there are essentially three queues associated with the passengers, which correspond to three different sequences $\{k\}$, $\{p_k\}$, and $\{d_k\}$, where $\{k\}$ corresponds to the sequence of requests, $\{p_k\}$ and $\{d_k\}$ correspond to those of pickups and drop-offs, respectively. The sequences $\{p_k\}$ and $\{d_k\}$ are permutations of $\{k\}$ in the extreme case where each cluster consists of only one event. In general, $\{p_k\}$ and $\{d_k\}$ include both permutations of subsets of $\{k\}$ and repetitions, as the orders of different pickups/drop-offs may be tied if in the same cluster.

\begin{definition}[Departure Constraint]
\label{def:departure}
\textit{Departure constraints} specify that for those clusters which contain pickup events, the shuttle should not depart from the routing point before the associated passengers arrive and get on board. 
\end{definition}

In our problem formulation, legitimate, capacity and MPS constraints are named \textit{primary constraints} as they affect the feasibility of the clustering pattern and the sequence of the clusters to be visited. While the departure constraints are called \textit{secondary constraints} which can always be satisfied by means of forcing the shuttle to wait at the routing points and adjust the departure times, as long as the three sets of primary constraints have all been met.

\subsection{Master Problem}
\label{subsec:master}

To serve a given set of requests, the dynamic shuttle service needs to determine: i) the clustering pattern, ii) the sequence of clusters to be visited, and iii) the routing point of each cluster. Then the shuttle will start from the depot $O$, visit each routing point exactly once in the right order, and end at a feasible terminal point. In this manuscript, we will not be covering how to determine the clustering pattern, but rather briefly address two main features while designing clusters as follows:
\begin{enumerate}
	\item A clustering pattern should be \textit{feasible}. \\
	A clustering pattern is feasible if and only if there exists at least one feasible sequence of the corresponding clusters that satisfies all the three sets of primary constraints. It should be noted that the primary constraints are originally defined for individual passengers, once clusters are formed, they are transferred to clusters, i.e., sets of events associated with individual passengers. For example, legitimate constraints are transferred in the way that $\forall k \in Z_1$, the cluster which contains $E_k^p$ should be visited before that contains $E_k^d$. Capacity and MPS constraints can be transferred as well. Feasibility is a \textit{hard requirement} for clustering patterns.
	\item A clustering pattern should be \textit{favorable}. \\ 
	Feasibility is necessary for a clustering pattern, however not sufficient. The reason is that though clustering helps reduce the number of stops for a trip serving a given set of passengers, which is likely to reduce the cost of the system, it may also possibly rule out certain favorable sequences or routing points due to extra restrictions resulted from the binding of different events (Fig. \ref{fig:favorable}). Therefore, a good clustering pattern that can be in use should also be favorable in the sense that it has potential to result in a reduction on $\mathcal C$ defined in (\ref{eqn:objfundef}), through a suitable design of the sequence and routing points. Though hard to define in a rigorous manner, the main philosophy behind a favorable clustering pattern is that it groups events which are not only close in space, i.e., share common intersections, but also close in the optimal sequence if no clustering is applied. Favorability is a \textit{soft requirement} for clustering patterns.  
\end{enumerate}
%Throughout the rest of the manuscript, clustering patterns will be explicitly given, based on which the determination of the sequence and routing points will be explored.
\begin{figure}[h!]
	\centering
	\includegraphics[width=1\textwidth]{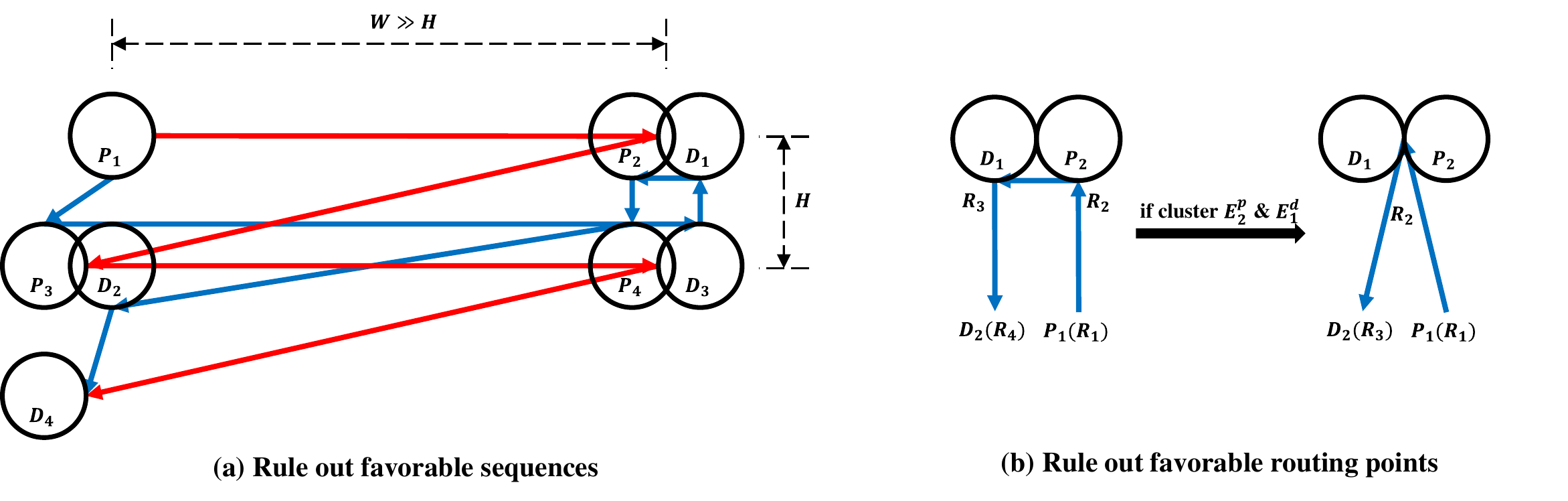}
	\caption{Scenarios where certain favorable sequences or routing points are ruled out due to clustering. (a): The blue route outperforms the red one, which is the only feasible route if cluster $\{E_2^p, E_1^d\}$, $\{E_3^p, E_2^d\}$ and $\{E_4^p, E_3^d\}$, in terms of the total travel distance of the shuttle. (b) The route on the right is worse in terms of the total travel distance, if $\{E_2^p, E_1^d\}$ are grouped therefore there is only one feasible routing point for the cluster.}
	\label{fig:favorable}
\end{figure} 

In what follows, we will assume that a feasible and favorable clustering pattern is explicitly given, based on which the determination of the sequence and routing points will be explored in order to minimize the cost defined in (\ref{eqn:objfundef}). In order to formally state this optimization problem, we construct a $N$ dimensional vector $S$ to represent the sequence of the clusters to be visited, where $S$ is a permutation of $[1, 2, \dots, N]$, $S(i)$ indicates the order of cluster $C_i$ in the visiting queue of the shuttle, $\forall i \in Z_2$. Denote $C_0=\emptyset$ as a dummy cluster representing the depot, whose corresponding space window is $S_W(O, 0)$ and $A_0=O$, $R_0=O$. Denote $R=(R_0, R_1, \dots, R_N)$. With the objective function specified as in (\ref{eqn:objfundef}), and the four sets of constraints as in Definitions 2.7 through 2.10, the \textit{master problem} for the rest of the manuscript is formally stated as  
\begin{equation}
\label{eqn:masterproblem}
\operatorname*{min}_{(S, R)\in S_f\times R_f} \mathcal C(S, R)
\end{equation}
where $S_f$ and $R_f$ are the feasible domains of the decision variables $S$ and $R$, respectively, that satisfy all of the constraints listed in Section \ref{subsec:cons}, and $R_f=A_0\times A_1\times \dotsb \times A_N$. 
%The master problem is essentially a generalized Traveling Salesman Problem (TSP) over $N$ areas, instead of points. 

\section{A Mixed Integer Quadratically Constrained Programming Formulation}
\label{sec:miqcp}
The underlying master problem is to minimize the time cost of the system defined in (\ref{eqn:objfundef}) subject to the four sets of constraints specified in Definitions 2.7 through 2.10. This has been compactly restated as the optimization problem in (\ref{eqn:masterproblem}). Given the discrete nature of $S$ and the continuous nature of $R$, it follows that the problem is one of a MIP. As the Euclidean distance metric is exploited, the formulation includes quadratic constraints as well. A natural next step therefore is to evaluate the use of a standard MIQCP formulation to solve the master problem. An MIQCP formulation is a model of an optimization problem that exploits both continuous and discrete decision variables and has constraints containing quadratic terms \cite{berthold2012extending}. In this section, we carry out such an evaluation.
%In this section, we propose a MIQCP formulation to fully capture the master problem defined in (\ref{eqn:masterproblem}), which can be solved via several standard solvers, e.g., Gurobi is used in this manuscript, to provide benchmarks to evaluate the performances of the AltMin algorithm which is detailed in Section \ref{sec:altmin}.

Given that the shuttle starts the trip from the depot $O$ and ends at a terminal cluster, in addition to $C_0$ denoting the depot, we add another dummy cluster $C_{N+1}$ to denote the virtual last stop. This is to ensure that the problem has an open-ended setting so that the shuttle is eligible to terminate the trip at any cluster as long as it is feasible. The travel time from the routing point of any other cluster to that of $C_{N+1}$ is set to be zero and the shuttle is forced to end the trip with $C_{N+1}$. The associated $A_{N+1}$ and $R_{N+1}$ have no actual physical meaning.

%For simplicity, denote the two sets we will be frequently using in the MIQCP formulation as
In addition to $Z_1$ and $Z_2$, we introduce $Z_3$ and $Z_4$ to facilitate the formulation as follows.
\begin{enumerate}
	\item $Z_3:=\{(i, j) \, | \, i \in \{0\} \cup Z_2, j \in Z_2 \cup \{N+1\}, \, i \neq j, \, \text{and} \, (i, j) \neq (0, N+1) \}$
	\item $Z_4:=\{(i, j) \, | \, i \in \{0\} \cup Z_2, j \in Z_2, \, \text{and} \, i \neq j\}$
\end{enumerate}

For each pair of clusters $(i, j) \in Z_3$, let $x_{ij}=1$ if the shuttle travels from $A_i$ to $A_j$, otherwise $x_{ij}=0$. $t_{ij}$ denotes the travel time of the shuttle from $R_i$ to $R_j$. For each $i \in \{0\} \cup Z_2 \cup \{N+1\}$, $T_i$ denotes the departure time of the shuttle from $R_i$, and $T_0=0$ is set as the departure time of the trip from $O$. $q_i$ denotes the number of passengers on board after the shuttle visits cluster $C_i$. It is clear that $q_0=q_{N+1}=0$. $f_i^p$ and $f_i^d$ denote the number of finished pickups and drop-offs before the shuttle visits cluster $C_i$, respectively. It follows that $f_0^d=f_0^d=0$ and $f_{N+1}^p=f_{N+1}^d=n$. Let $\ell_i^p$, $\ell_i^d$ and $\ell_i$ denote the number of pickups, drop-offs and net loads of cluster $C_i$, respectively, where net load denotes the difference between the number of pickups and that of drop-offs. That is, $\ell_i=\ell_i^p-\ell_i^d$, with $\ell_0^p=\ell_0^d=\ell_0=0$ and $\ell_{N+1}^p=\ell_{N+1}^d=\ell_{N+1}=0$. Moreover, for each $k \in Z_1$, $w_k^p$ and $w_k^d$ denote the walking times of passenger $k$ before being picked up and after being dropped off, respectively, and $c\ell_k^p, c\ell_k^d$ denote the indexes of the clusters which contain $E_k^p$ and $E_k^d$, respectively. It should be noted that $\ell_i^p, \ell_i^d, \ell_i, c\ell_k^p, c\ell_k^d, \forall i \in \{0\} \cup Z_2 \cup \{N+1\}, k \in Z_2$ are constants associated with the clustering pattern.      

%Therefore the master problem could be formulated as an MIQCP formulation using the notations defined above
With the above definitions, the master problem in (\ref{eqn:masterproblem}) can be restated as an MIQCP formulation
\begin{equation}
\label{eqn:objfundefmiqcp}
\operatorname*{Min} \mathcal C=\gamma_1 T_{N+1} + \gamma_2 \sum_{k=1}^{n} (\alpha_1 T_{c\ell_k^p} + \alpha_2 (T_{c\ell_k^d}-T_{c\ell_k^p}) + \alpha_3^p w_k^p + \alpha_3^d w_k^d)
\end{equation}

subject to
\begin{align}
\sum_{j \in Z_2} x_{0j} &= 1 \label{eqn:conmiqcp_1} \\
\sum_{i \in \{0\} \cup Z_2} x_{ij} &= 1 &&\forall j \in Z_2 \label{eqn:conmiqcp_2} \\
\sum_{j \in Z_2 \cup \{N+1\}} x_{ij} &= 1 &&\forall i \in Z_2 \label{eqn:conmiqcp_3} \\
\sum_{i \in Z_2} x_{i(N+1)} &= 1 \label{eqn:conmiqcp_4} \\
x_{ij}&\in \{0, 1\} &&\forall (i, j) \in Z_3 \label{eqn:conmiqcp_5} \\
T_j &\geq (T_i+t_{ij}+t_a)x_{ij} &&\forall (i,j) \in Z_3 \label{eqn:conmiqcp_6} \\
T_{c\ell_k^d} &\geq T_{c\ell_k^p}+ t_{(c\ell_k^p)(c\ell_k^d)}+t_a &&\forall k \in Z_1 \label{eqn:conmiqcp_7} \\
c &\geq q_i &&\forall i \in \{0\} \cup Z_2 \cup \{N+1\} \label{eqn:conmiqcp_8} \\
q_j &\geq (q_i+\ell_j)x_{ij} &&\forall (i, j) \in Z_3 \label{eqn:conmiqcp_9} \\
f_j^p &\geq (f_i^p+\ell_i^p){x_{ij}} &&\forall (i, j) \in Z_3 \label{eqn:conmiqcp_10} \\
f_j^d &\geq (f_i^d+\ell_i^d){x_{ij}} &&\forall (i, j) \in Z_3 \label{eqn:conmiqcp_11} \\
|(f_{c\ell_k^p}^p+1)-k| &\leq MPS^p &&\forall k \in Z_1 \label{eqn:conmiqcp_12} \\
|(f_{c\ell_k^d}^d+1)-k| &\leq MPS^d &&\forall k \in Z_1 \label{eqn:conmiqcp_13} \\
R_i &\in A_i &&\forall i \in \{0\} \cup Z_2 \label{eqn:conmiqcp_14} \\
{v_s}^2{t_{ij}}^2&=||R_i-R_j||_2^2 &&\forall (i, j) \in Z_4 \label{eqn:conmiqcp_15} \\
t_{i(N+1)}&=0 &&\forall i \in Z_2 \label{eqn:conmiqcp_16} \\
{v_p}^2{w_k^p}^2&=||P_k-R_{c\ell_k^p}||_2^2 &&\forall k \in Z_1 \label{eqn:conmiqcp_17} \\
{v_p}^2{w_k^d}^2&=||D_k-R_{c\ell_k^d}||_2^2 &&\forall k \in Z_1 \label{eqn:conmiqcp_18} \\
w_k^p &\leq T_{c\ell_k^p} &&\forall k \in Z_1 \label{eqn:conmiqcp_19} 
\end{align}

Equation (\ref{eqn:objfundefmiqcp}) is a restatement of the objective function in the MIQCP formulation, which is equivalent to that in (\ref{eqn:objfundef}). $T_{N+1}$ corresponds to the departure time from cluster $C_{N+1}$, which equals the finishing time of the trip. $T_{c\ell_k^p}$ represents the difference between the pickup time and the request time therefore equals the waiting time of passenger $k$. Similarly, $T_{c\ell_k^d}-T_{c\ell_k^p}$ represents the difference between the drop-off time and the pickup time and hence equals the $k$th passenger's riding time.

Constraints (\ref{eqn:conmiqcp_1})-(\ref{eqn:conmiqcp_4}) ensure that the shuttle departs from the depot $O$, ends the trip at the dummy cluster $C_{N+1}$, and visits every other cluster exactly once. Consistence of the departure times is guaranteed by constraints (\ref{eqn:conmiqcp_6}), which essentially serve as the subtour elimination conditions. Equations (\ref{eqn:conmiqcp_7}) and (\ref{eqn:conmiqcp_8}) ensure that the legitimate constraints and capacity constraints are satisfied. Equations (\ref{eqn:conmiqcp_9})-(\ref{eqn:conmiqcp_11}) guarantee the consistence of the numbers of passengers on board, finished pickups and drop-offs at anytime of the trip, respectively. Equations (\ref{eqn:conmiqcp_12}) and (\ref{eqn:conmiqcp_13}) ensure that the MPS constraints are obeyed. Equations (\ref{eqn:conmiqcp_14}) essentially guarantee that each routing point is within the corresponding area. Equations (\ref{eqn:conmiqcp_15}) and (\ref{eqn:conmiqcp_16}) define the travel time of the shuttle between each pair of routing points, and Equations (\ref{eqn:conmiqcp_17}) and (\ref{eqn:conmiqcp_18}) defines the walking times of the passengers before being picked up and after being dropped off, respectively. Equations (\ref{eqn:conmiqcp_19}) ensure that the shuttle departs from each area after all passengers whose pickups are associated have already arrived at the corresponding routing point and gotten on board. 

This formulation is nonlinear because of constraints (\ref{eqn:conmiqcp_6}), (\ref{eqn:conmiqcp_9})-(\ref{eqn:conmiqcp_11}), (\ref{eqn:conmiqcp_14}), (\ref{eqn:conmiqcp_15}), (\ref{eqn:conmiqcp_17}) and (\ref{eqn:conmiqcp_18}). Of these, we note that (\ref{eqn:conmiqcp_6}) and (\ref{eqn:conmiqcp_9})-(\ref{eqn:conmiqcp_11}) are intrinsically linear and can be simplified into equivalent linear formats by introducing extra explaining constants. For example, the subtour elimination constraints in (\ref{eqn:conmiqcp_6}) are equivalent to
\begin{equation}
\label{eqn:subtourmiqcp}
T_j \geq T_i + t_{ij}+t_a-T_{ij}(1-x_{ij})
\end{equation}
for each $(i,j) \in Z_3$, where $T_{ij} \geq T_i+t_{ij}+t_a$ and $T_j \geq 0$. Simply set $T_{ij}=T_c$ where $T_c \geq \underset{(i, j) \in Z_3}{\text{max}} \, (T_i+t_{ij}+t_a)$ in the implementation. This technique is very similar to that used in the TSP formulation in \cite{miller1960integer, cordeau2006branch}. Using the same procedure, (\ref{eqn:conmiqcp_9})-(\ref{eqn:conmiqcp_11}) can be linearized as follows
\begin{equation}
\label{eqn:loadmiqcp}
q_j \geq q_i+\ell_j- Q_{ij}(1-x_{ij})
\end{equation}
\begin{equation}
\label{eqn:pickmiqcp}
f_j^p \geq f_i^p+\ell_i^p-F_{ij}^p(1-x_{ij})
\end{equation}
\begin{equation}
\label{eqn:dropmiqcp}
f_j^d \geq f_i^d+\ell_i^d-F_{ij}^d (1-x_{ij})
\end{equation}
%\begin{equation}
%\label{eqn:dropmiqcp}
%f_j^p \geq f_i^p+l_i^p-F_{ij}^p(1-x_{ij}), \quad f_j^d \geq f_i^d+l_i^d-F_{ij}^d (1-x_{ij})
%\end{equation}
where $Q_{ij} \geq q_i+\ell_j$ and $q_j \geq 0$. Similarly set $Q_{ij}=Q_c$ where $Q_c \geq \underset{(i, j) \in Z_3}{\text{max}} \, (q_i+\ell_j)$. And $F_{ij}^p \geq f_i^p+\ell_i^p$, $F_{ij}^d \geq f_i^d+\ell_i^d$, $f_j^p, f_j^d \geq 0$. Set $F_{ij}^p=F_{ij}^d=F_c$ where $F_c \geq \text{max} \bigg\{ \underset{i \in Z_2}{\text{max}} \, (f_i^p+\ell_i^p), \underset{i \in Z_2}{\text{max}} \, (f_i^d+\ell_i^d) \bigg\}$.

The remaining constraints (\ref{eqn:conmiqcp_14}), (\ref{eqn:conmiqcp_15}), (\ref{eqn:conmiqcp_17}), and (\ref{eqn:conmiqcp_18}) contain nonlinearities which stem from the inherent quadratic nature of the Euclidean distance metric which cannot be linearized unless via approximation. Among them, (\ref{eqn:conmiqcp_14}) are convex as the space window for each pickup/drop-off event is a circle which is convex, and therefore the intersection of any combination of the original space windows is convex as well. (\ref{eqn:conmiqcp_15}), (\ref{eqn:conmiqcp_17}), and (\ref{eqn:conmiqcp_18}) are non-convex due to the equality sense in the definitions of the time cost terms. Typical methods such as convex relaxation with Second Order Cone (SOC) constraints can be utilized to accommodate the non-convexity. One such approach is detailed below.
For each $(i, j) \in Z_3$ and each $k \in Z_2$, introduce rotational explaining variables 
\begin{equation}
\label{eqn:rotationmiqcp}
R_{ij}^r=\frac{R_i-R_j}{v_s}
\end{equation}
and shifted explaining variables
%\begin{equation}
%\label{eqn:shiftpickmiqcp}
%R_k^{s, p}=\frac{R_{cl_k^p}-P_k}{v_p}
%\end{equation}
%\begin{equation}
%\label{eqn:shiftdropmiqcp}
%R_k^{s, d}=\frac{R_{cl_k^d}-D_k}{v_p}
%\end{equation}
\begin{equation}
\label{eqn:shiftmiqcp}
R_k^{s, p}=\frac{R_{c\ell_k^p}-P_k}{v_p}, \quad R_k^{s, d}=\frac{R_{c\ell_k^d}-D_k}{v_p}
\end{equation}
Therefore the SOC formulations of (\ref{eqn:conmiqcp_15}), (\ref{eqn:conmiqcp_17}), and (\ref{eqn:conmiqcp_18}) are
\begin{equation}
\label{eqn:socrotationmiqcp}
||R_{ij}^r||_2^2 \leq {t_{ij}}^2
\end{equation}
%\begin{equation}
%\label{eqn:socshiftpickmiqcp}
%||R_k^{s, p}||_2^2 \leq {w_k^p}^2
%\end{equation}
%\begin{equation}
%\label{eqn:socshiftdropmiqcp}
%||R_k^{s, d}||_2^2 \leq {w_k^d}^2
%\end{equation}
\begin{equation}
\label{eqn:socshiftmiqcp}
||R_k^{s, p}||_2^2 \leq {w_k^p}^2, \quad ||R_k^{s, d}||_2^2 \leq {w_k^d}^2
\end{equation}
where $t_{ij} \geq 0$, and $w_k^P, w_k^d \geq 0$.

Such an SOC based relaxation approach becomes exact during the process of the minimization of (\ref{eqn:objfundefmiqcp}), which guarantees that the SOC constraints (\ref{eqn:socrotationmiqcp})-(\ref{eqn:socshiftmiqcp}) are equivalent to the original definitions of the time cost terms in (\ref{eqn:conmiqcp_15}), (\ref{eqn:conmiqcp_17}), and (\ref{eqn:conmiqcp_18}). The overall dimensionality of the decision variables and constraints of the MIQCP formulation are summarized in Table \ref{tab:dimensionmiqcp}.
\begin{table}[]
	\centering
	\caption{Dimensionality of the decision variables and constraints of the MIQCP formulation.}
	\label{tab:dimensionmiqcp}
	\begin{tabular}{ccr}
		\toprule
		\multicolumn{1}{c}{}        & \multicolumn{1}{c}{Category} & \multicolumn{1}{c}{Dimensionality} \\ \midrule
		\multirow{3}{*}{Variables}   & Binary                        & $N^2+N$                             \\
		& Continuous                    & $3N^2+7N+6n+10$                     \\ \cmidrule{2-3} 
		& Sum                           & $4N^2+8N+6n+10$                     \\ \midrule
		\multirow{4}{*}{Constraints} & Quadratic                     & $2n$                                \\
		& SOC                           & $N^2+2n$                            \\
		& Linear                        & $6N^2+7N+10n+11$                    \\ \cmidrule{2-3} 
		& Sum                           & $7N^2+7N+14n+11$                    \\ 
		\bottomrule
	\end{tabular}
\end{table}
We note that if Manhattan distance or metrics that correspond to real street topologies rather than a Euclidean distance metric is used, the problem can be further reduced into an MIP formulation with linear constraints.
%With the above SOC relaxation, we now have the constraints (\ref{eqn:conmiqcp_14}), (\ref{eqn:conmiqcp_15}), (\ref{eqn:conmiqcp_17}), and (\ref{eqn:conmiqcp_18}) in a convex form. 
%The optimal solution is essentially the one among the finite number of optimal sequence-routing points pair which result in the minimal value of system cost.

Given that the master problem in (\ref{eqn:masterproblem}) is now formulated to a standard MIQCP, commercial solvers, such as Gurobi in our implementation, can be used to find the optimal solution with sufficient computational time. This is because the number of feasible sequences is finite, and Gurobi basically traverses all possible sequence and computes the corresponding optimal routing points. However, the computational complexity is huge as both the dimensionality of the decision variable as well as constraints are quadratic in the size of the problem $(n, N)$. For example, if 10 passengers are to be served and 16 clusters are formed, there are 1,222 decision variables in total, 272 of which are binary. The total number of constraints is 2,055 and 20 of them are quadratic, 276 are SOC and the rest are linear. It can be computationally highly burdensome, and necessitates the development of a few heuristics. This is the subject of Section \ref{sec:altmin}.

\section{An Alternating Minimization Algorithm}
\label{sec:altmin}

We propose an AltMin algorithm in this section to solve the master problem formulated in (\ref{eqn:masterproblem}). Due to its simplicity and effectiveness, AltMin has long been a popular optimization method, dating back to early works in the optimization literature (e.g., \cite{ortega2000iterative}), and has been widely studied under various assumptions (e.g., \cite{auslender1976optimisation, luo1993error}). It consists of a block coordinate descent approach, and operates by adjusting one block of coordinates at a time so as to attain successive reductions in the objective function. Given the presence of two sets of decision variables $S$ and $R$ in the master problem under consideration, and given that they are quite distinct from another, the use of such an approach is fairly natural. We therefore develop an AltMin algorithm with two phases, where in Phase 1, we optimize $\mathcal C$ over $S$, keeping $R$ fixed, while in Phase 2, we keep $S$ fixed and optimize $\mathcal C$ over $R$. A combination of consecutive Phase 1 plus Phase 2 constitutes one \textit{mega iteration} in our AltMin algorithm.

More formally, let $(S^\ast, R^\ast)$ denotes the desired decision variables that optimizes $\mathcal C$ in (\ref{eqn:masterproblem}). That is
\begin{equation}
\label{eqn:masterproblemargmin}
(S^\ast, R^\ast)=\operatorname*{argmin}_{(S, R)\in S_f\times R_f} \mathcal C(S, R)
\end{equation}
It should be noted that $(S^\ast, R^\ast)$ can be derived exactly via the MIQCP formulation in Section \ref{sec:miqcp}.
The AltMin algorithm is proposed to solve $(S^\ast, R^\ast)$ approximately via alternately iterating Phase 1 and Phase 2, where at the $h$th mega iteration, $\forall h \in \mathbb Z_{>0}$, Phases 1 and 2 correspond to updates
\begin{equation}
\label{eqn:step1}
S^h=\operatorname*{argmin}_{S\in S_f} \, \mathcal C(S, R^{h-1})
\end{equation}
\begin{equation}
\label{eqn:step2}
R^h=\operatorname*{argmin}_{R\in R_f} \, \mathcal C(S^h, R)
\end{equation}
respectively, where $R^0$ is the initialization of the routing points. The mega iterations are terminated with a suitable stopping criterion.

The advantages of the AltMin algorithm are mainly that it decomposes the master problem, which is complicated and extremely computationally costly, into two subproblems, which are both fairly easy to solve. And in many circumstances, it is powerful in terms of optimality and computational efficiency \cite{csisz1984information, niesen2009adaptive, wang2008new}. Specifically, in our routing task, compared with the MIQCP formulation which seeks to find the optimal solution by traversing all feasible sequences, the AltMin algorithm restricts its search to a subset of the entire solution space, which is determined judiciously through successive Phase 1-Phase 2 optimizations.

One possible drawback is that the decoupling process of $S$ and $R$ results in the intrinsic heuristic nature of AltMin. The convergence and optimality properties of the generic AltMin algorithm can be found in various references such as \cite{jain2013low, grippo2000convergence, tseng2001convergence, lin2007projected}, which follow under certain regularity conditions. However, the mixed presence of both discrete and continuous decision variables introduces significant challenges in providing similar analytical guarantees for convergence and optimality. Instead, in this manuscript, we focus on a numerical demonstration of the superiority of AltMin over MIQCP using real operational data over a range of scenarios. Before proceeding to this demonstration in Section \ref{sec:experiments}, we discuss details of Phases 1 and 2 of the proposed AltMin algorithm in this section.

To start the AltMin algorithm, the routing points need to be initialized as $R^0$. For simplicity and consistency, each element $R^0_i$ of $R^0$ is computed as the center of mass of $A_i, \forall i \in \{0\} \cup Z_2$, in our implementation. With this initialization, we discuss details of Phase 1 in Section \ref{sub:phase1}, details of Phase 2 in Section \ref{sub:phase2}, convergence analysis and stopping criterion in Section \ref{subsec:stop}.

\subsection{Phase 1}
	\label{sub:phase1}
As mentioned above, in Phase 1, the routing points $R$ are fixed, and the focus is on the optimization of $\mathcal C$ over $S$. In our problem formulation, the shuttle is assumed to travel between the areas directly following the straight lines connecting the corresponding routing points, therefore the statuses of the system of when the shuttle is at each cluster are sufficient to determine the sequence of visit. The status of the system when the shuttle is at each cluster is uniquely determined by: 1) the cluster it is current at, and 2) the other clusters that have been visited. With this assumption, we define the state vectors of the system, and derive the underlying recurrence relation that corresponds to Phase 1 using these states. This is described in Sections \ref{subsubsec:vector} through Section \ref{subsubsec:recurrence}.
\subsubsection{State Vector}
\label{subsubsec:vector}
Let an $N+1$ dimensional vector $(L, e_1, e_2, \dots, e_N)$ denote the state of the system where
\begin{itemize}
	\item $L \in \{0\} \cup Z_2$ is the index of the cluster that the shuttle is currently at. $L=0$ corresponds to the shuttle being at the depot, and $1 \leq L \leq N$ corresponds to the shuttle being at cluster $C_L$.
	\item $e_i \in \{0, 1\}, \forall i \in Z_2$ is a flag that indicates whether the shuttle has visited cluster $C_i$ or not, with $e_i=1$ if the shuttle has visited or is currently at cluster $C_i$, and $e_i=0$ otherwise.
\end{itemize}
The definition of the state vector implies that the \textit{initial state} of the shuttle could only be $(0, 0, 0, \dots, 0)$. And due to the open-ended setting of our formulation, the shuttle is eligible to finish the trip with any feasible cluster except for $C_0$. Therefore, a \textit{terminal state} is expressed as $(L, 1, 1, \dots, 1), \forall L \in Z_2$.

\subsubsection{Feasibility of State Vectors} 
\label{subsubsec:vectorfeasibility}
A state vector is feasible if and only if: 1) it is consistent by definition, 2) it satisfies all primary constraints specified in Section \ref{subsec:cons}. Each of these cases is detailed below.

\textit{4.2.2.1 Consistency}

A state vector should be consistent by definition, essentially: if $L=0$, it follows that $e_i=0, \forall i \in Z_2$; otherwise if $1 \leq L \leq N$, it follows that $e_L=1$.

\textit{4.2.2.2 Capacity Constraints}

Recalling for each $i \in Z_2$, $\ell_i$ is the net load, which corresponds to the difference between the number of passengers been picked up and that of passengers been dropped off, at cluster $C_i$. It follows that capacity constraints are satisfied if
\begin{equation}
\label{eqn:capacitystate}
\sum_{i \, | \, e_i=1, \, i \in Z_2} \ell_i \leq c
\end{equation}
The left hand side of (\ref{eqn:capacitystate}) computes the sum of net loads of the shuttle from departure until the current state, therefore equates to the number of passengers who are currently on board.

\textit{4.2.2.3 MPS Constraints}

$\sum_{i \, | \, e_i=1, \, i \in Z_2 \setminus \{L\}} \ell_i^p$ and $\sum_{i \, | \, e_i=1, \, i \in Z_2 \setminus \{L\}} \ell_i^d$ correspond to the number of total finished pickups and drop-offs prior to visiting the current cluster $C_L$, respectively. Therefore, the MPS constraints can be verified as follows
\begin{equation}
\label{eqn:mpspstate} 
\Bigg| \bigg(\sum_{i \, | \, e_i=1, \, i \in Z_2 \setminus \{L\}} \ell_i^p + 1 \bigg) - k \Bigg| \leq MPS^p, \quad \forall k \, | \, E_k^p \in C_L
\end{equation}
\begin{equation}
\label{eqn:mosdstate}
\Bigg| \bigg(\sum_{i \, | \, e_i=1, \, i \in Z_2 \setminus \{L\}} \ell_i^d + 1 \bigg) - k \Bigg| \leq MPS^d, \quad \forall k \, | \, E_k^d \in C_L
\end{equation}

\textit{4.2.2.4 Legitimate Constraints}

The legitimate constraints are accommodated via the fact that a state is feasible only if there exist at least one feasible \textit{next state}, as long as it is not a terminal state. A next state of a given state is defined as one that is eligible to occur immediately after the current state. A state vector $(L', e_1', e_2', \dots, e_N'), L' \neq 0$, where $L'$ is the index of the \textit{next cluster} of the current state, can be a next state if and only if
\begin{enumerate}
	\item The cluster $C_{L'}$ has not been visited yet, i.e., $e_{L'}=0$. This condition will automatically guarantee that all passengers to be picked up/dropped off have not been picked up/dropped off yet, i.e., the events contained in $C_{L'}$ have not yet occurred.
	\item The passengers to be dropped off have already been picked up, i.e., $\forall E_{k_d}^d \in C_L, \, e_{c\ell_{k_d}^p}=1$.
	\item The status of the clusters will remain the same except for that of cluster $C_{L'}$, i.e., $e_{L'}=1$ while $e_i'=e_i, \forall i \in Z_2 \, \text{and} \, i \neq L'$. 
\end{enumerate}
This screening of feasible next states guarantees that the legitimate constraints of all passengers are satisfied. The secondary constraints, which are defined in Definition 2.10, are not considered in Phase 1 as they does not affect the feasibility of the sequence of clusters to be visited.  

\subsubsection{Recurrence Relation} 
\label{subsubsec:recurrence}
Once the feasibility verifications of all possible states have been conducted, the next step is optimality considerations, which is accomplished via a backward recursion. Let $V(L, e_1, e_2, \dots, e_N)$ be the optimal system cost measured in terms of $\mathcal C$ defined in (\ref{eqn:objfundef}) of all subsequent decisions from the current state $(L, e_1, e_2, \dots, e_N)$ onwards till the end of the trip, i.e., any feasible terminal state is reached. It should be noted that $V$ is only defined for the states which are feasible. It is clear that $\mathcal C^\ast=V(0, 0, 0, \dots, 0)$ and $V=0$ for any terminal state. The recurrence relation of $V(L, e_1, e_2, \dots, e_N)$ is therefore given by
\begin{equation} 
	\label{eqn:recurrence}
	V(L, e_1, e_2, \dots, e_N)
	=\left\{ 
	\begin{array}{cl} 
		\underset{\forall L'}{\text{min}} \Big[ t(L, L') \cdot M + V(L', e_1', e_2', \dots, e_N') \Big] &\text{if not a terminal state}\\
		0 &\text{otherwise}
	\end{array} \right.
\end{equation}
The term $t(L, L') \cdot M$ denotes the subcomponent of $\mathcal C$ that corresponds to the incremental time cost incurred due to the travel of the shuttle from cluster $C_L$ to $C_{L'}$. It is distributed over the first, second, and third term in (\ref{eqn:objfundef}), since it includes time costs incurred by the shuttle, the passengers waiting to be picked up, and those that are riding on the shuttle. Therefore $M$ is given by 
\begin{equation}
\label{eqn:minbhk}
M=\gamma_1+\gamma_2 \Bigg( \alpha_1\sum_{i \, | \, e_i=0, \, i \in Z_2} \ell_i^p + \alpha_2\sum_{i \, | \, e_i=1, \, i \in Z_2} \ell_i \Bigg)
\end{equation}
It should be noted that $t(L, L')=\frac{||R_L-R_{L'}||_2}{v_s}+t_a$ is the time taken for the shuttle to travel from $C_L$ to $C_{L'}$. It is possible to calculate $t(L, L')$ off line as the routing points are fixed prior to this phase.

\subsubsection{Remarks}
Several remarks regarding Phase 1 are discussed as follows:
\begin{enumerate}
	\item The sequence of the clusters to be visited consists of a series of sequential transitions of feasible system states. Let $\mathcal{E}=\sum_{i \in Z_2} e_i$ denote the \textit{entropy} of the state $(L, e_1, e_2, \dots, e_N)$. One can see that $\mathcal{E}=0$ in the initial state, increased by one at each state transition and ending with $\mathcal{E}=N$ at any feasible terminal state. In the implementation of Phase 1, the optimization relation between adjacent states in (\ref{eqn:recurrence}) only occurs at states that are feasible. (\ref{eqn:recurrence}) is implemented as a backward recursion which starts from all feasible terminal states, decreases the entropy by one at each step, and finds the optimal next state for every feasible state, until the initial state is reached. The optimal next state for all feasible states is memorized in this manner. Together, the overall AltMin algorithm starts from the initial state, traces forward through successive optimal next states, and puts together the overall optimal route.
	\item \textit{Computational Complexity}. By definition, there are at most  $(N+1) 2^N$ possible states to begin with, which are decreased to $1+N 2^{N-1}$ after a consistency check. For each of these states, the further feasibility check based on the primary constraints and the optimization screening on the therefore feasible states both take $O(n)$ time. Hence, the overall running time is $O(n N 2^{N-1})$. Due to the presence of the primary constraints, the number of feasible states is in fact much smaller than $1+N2^{N-1}$, though this is not able to improve the asymptotic running time of Phase 1, it helps accelerate the algorithm significantly in actual computational experiments. In particular, the algorithm adopts a Dynamic Programming (DP) paradigm which provides an exact solution for Phase 1, though the computational complexity is exponential with respect to the problem size, it is asymptotically better than classical DP approaches \cite{psaraftis1980dynamic}. 
	\item It should be noted that the problem addressed here in Phase 1 is a generalized TSP. The specific approach used is akin to the Bellman-Held-Karp (BHK) algorithm \cite{psaraftis1980dynamic}. Certain modifications have to be made, however, in order to address the fact that the algorithm has to derive the sequencing of sets of events, i.e. clusters, rather than sequencing single events. Moreover, the overall master problem is a even more generalized TSP, which is defined over areas, instead of points.
	\item The modified BHK algorithm we have proposed is suitable for shuttles with capacity up to 15. If even larger capacities need to be considered, one can either distribute the implementation of the modified BHK algorithm, e.g., system states with the same entropy have no dependency, therefore their feasibility screenings and optimization considerations could both be parallelized to speed up the computation, or certain heuristics can be developed. %which remains one of our main future directions.  
	\item The main analytical property of the solution derived via the modified BHK algorithm to Phase 1 is summarized in Lemma \ref{lemma:phase1}.
\end{enumerate}

\begin{lemma}[Uniqueness and Optimality of Phase 1]
	\label{lemma:phase1}
	$\forall h \in \mathbb Z_{\geq 0}$, given $R^h$, $S^{h+1}$ is uniquely determined in (\ref{eqn:step1}) via Phase 1 and it is optimal.  
\end{lemma}
\begin{proof}
	$\forall h \in \mathbb Z_{\geq 0}$, given $R^h$, $S^{h+1}$ is derived via the modified BHK algorithm detailed in Section \ref{sub:phase1}. The algorithm essentially examines all feasible states and seeks the sequence with minimal cost in a backward recursive manner, therefore $S^{h+1}$ is optimal and uniquely determined.  
\end{proof}
   
\subsection{Phase 2}
	\label{sub:phase2}
	
In Phase 2, the sequence of the clusters to be visited, $S$, is fixed, and the focus is on the optimization of $\mathcal C$ over the routing points, $R$, whose determination process can be formulated as a Quadratically Constrained Quadratic Programming (QCQP). The QCQP in Phase 2 is convex and can be solved to optimality via standard solvers such as  Gurobi in our implementation. It should be noted that the primary constraints need not to be examined in this phase as the sequence is given while the secondary constraints are to be accommodated. The QCQP formulation is detailed as follows.

For each cluster $C_i, \forall i \in \{0\} \cup Z_2$, its routing point should be within the corresponding area, i.e., $R_i \in A_i$. And $t_i, w_k^p, w_k^d$, $\forall i \in Z_2, k \in Z_1$ are defined as in Section \ref{sec:miqcp}. Similarly, rotational and shifted explaining variables are introduced to transform the quadratic equality representations of the time cost terms to SOC constraints as follows,
\begin{equation}
\label{eqn:qcqprev}
R_i^r=\frac{R_{S(i-1)}-R_{S(i)}}{v_s}, \quad \forall i \in Z_2
\end{equation}
%\begin{equation}
%\label{eqn:qcqpsevp}
%R_k^{s,p}=\frac{R_{cl_k^p}-P_k}{v_p}
%\end{equation}
%\begin{equation}
%\label{eqn:qcqpsevd}
%_k^{s,d}=\frac{R_{cl_k^d}-D_k}{v_p}
%\end{equation}
\begin{equation}
\label{eqn:qcqpsevd}
R_k^{s,p}=\frac{R_{c\ell_k^p}-P_k}{v_p}, \quad R_k^{s,d}=\frac{R_{c\ell_k^d}-D_k}{v_p}, \quad \forall k \in Z_1
\end{equation}
%For each cluster $C_i, \forall i \in \{0\} \cup S_2$, its routing point should be within the corresponding area, i.e., $R_i \in A_i$. Let $t_i$ denote the travel %time of the shuttle on the $i$th trip segment, therefore
%\begin{equation} 
%\label{eqn:qcqpt1} 
%{v_s}^2 {t_i}^2=||R_{S(i-1)}-R_{S(i)}||_2^2, \forall i \in S_2
%\end{equation}
%$w_k^p$ and $w_k^d$ are defined in the same way as in Section \ref{sec:miqcp}, therefore
%\begin{equation} 
%\label{eqn:qcqpwp1} 
%{v_p}^2 {w_k^p}^2=||R_{cl_k^p}-P_k||_2^2
%\end{equation}
%\begin{equation} 
%\label{eqn:qcqpwd1} 
%{v_p}^2 {w_k^d}^2=||R_{cl_k^d}-D_k||_2^2
%\end{equation}
%\begin{equation} 
%\label{eqn:qcqpw1} 
%{v_p}^2 {w_k^p}^2=||R_{cl_k^p}-P_k||_2^2, \quad {v_p}^2 {w_k^d}^2=||R_{cl_k^d}-D_k||_2^2
%\end{equation}
%
%Similarly, in order to have the formulation solvable in Gurobi, two more sets of explaining variables are introduced to transform the quadratic equality %constraints to SOC constraints. 
%
%Rotational explaining variables 
%\begin{equation}
%\label{eqn:qcqprev}
%R_i^r=\frac{R_{S(i-1)}-R_{S(i)}}{v_s}, \forall i \in S_2
%\end{equation}
%
%Shifted explaining variables, $\forall k \in S_1$
%\begin{equation}
%\label{eqn:qcqpsevp}
%R_k^{s,p}=\frac{R_{cl_k^p}-P_k}{v_p}
%\end{equation}
%\begin{equation}
%\label{eqn:qcqpsevd}
%_k^{s,d}=\frac{R_{cl_k^d}-D_k}{v_p}
%\end{equation}
%\begin{equation}
%\label{eqn:qcqpsevd}
%R_k^{s,p}=\frac{R_{cl_k^p}-P_k}{v_p}, \quad R_k^{s,d}=\frac{R_{cl_k^d}-D_k}{v_p}
%\end{equation}
%
Therefore,
\begin{equation}
\label{eqn:qcqpM}
||R_i^r||_2^2 \leq {t_i}^2, \quad \forall i \in Z_2
\end{equation}
%\begin{equation}
%\label{eqn:qcqpmp}
%||R_k^{s,p}||_2^2 \leq {w_k^p}^2
%\end{equation}
%\begin{equation}
%\label{eqn:qcqpmd}
%||R_k^{s,d}||_2^2 \leq {w_k^d}^2
%\end{equation}
\begin{equation}
\label{eqn:qcqpmd}
||R_k^{s,p}||_2^2 \leq {w_k^p}^2, \quad ||R_k^{s,d}||_2^2 \leq {w_k^d}^2, \quad \forall k \in Z_1
\end{equation}
where $t_i \geq 0$ and $w_k^p, w_k^d \geq 0$. In order to satisfy the departure constraints, the walking time of any passenger to be picked up should not exceed the departure time of the shuttle from the cluster that contains the corresponding pickup event, which is rewritten as
\begin{equation}
\label{eqn:qcqplazy}
w_k^p \leq \sum_{i=1}^{c\ell_k^p} t_i + S(c\ell_k^p)t_a, \quad \forall k \in Z_1
\end{equation}

Using the above notations, we rewrite (\ref{eqn:objfundef}) as
%\begin{equation}
%\label{eqn:qcqpobjfun}
%\begin{align}
%\label{eqn:qcqpobjfun}
%	\begin{split}
%C &= \sum_{i=1}^N \Big[ \gamma_1 + \gamma_2 \Big( \alpha_1 g_i^1 + \alpha_2 g_i^2 \Big) \Big] (t_i + t_a) + \gamma_2 \sum_{k=1}^n \Big( \alpha_3^p w_k^p + \alpha_3^d w_k^d \Big) \\
%  &= \sum_{i=1}^N \Big[ \gamma_1 + \gamma_2 \Big( \alpha_1 g_i^1 + \alpha_2 g_i^2 \Big) \Big] t_i + \gamma_2 \sum_{k=1}^n \Big( \alpha_3^p w_k^p + \alpha_3^d w_k^d \Big) + \delta
%	\end{split}
%\end{align}
\begin{equation}
\label{eqn:qcqpobjfun}
\mathcal C = \sum_{i=1}^N \Big[ \gamma_1 + \gamma_2 \Big( \alpha_1 g_i^1 + \alpha_2 g_i^2 \Big) \Big] (t_i + t_a) + \gamma_2 \sum_{k=1}^n \Big( \alpha_3^p w_k^p + \alpha_3^d w_k^d \Big)
\end{equation}
%\end{equation}
%where $\delta = \sum_{i=1}^N \Big[ \gamma_1 + \gamma_2 \Big( \alpha_1 g_i^1 + \alpha_2 g_i^2 \Big) \Big] t_a$ is a constant summarizing the total contributions of the extra time cost $t_a$ penalized at each station.
where
\begin{equation}
\label{eqn:waitqcqp}
g_i^1=n-\sum_{j=0}^{i-1} \ell_{S(j)}^p, \quad \forall i \in Z_2
\end{equation}
is the number of passengers who have not been picked up yet therefore are waiting when the shuttle is traveling on the $i$th trip segment, and 
\begin{equation}
\label{eqn:rideqcqp}
g_i^2=\sum_{j=0}^{i-1} \ell_{S(j)}, \quad \forall i \in Z_2 
\end{equation}
is the number of passengers who are on board when the shuttle is traveling on the $i$th trip segment.

With the above definitions, the underlying problem in Phase 2 is formulated as a QCQP which minimizes $\mathcal C$ in (\ref{eqn:qcqpobjfun}), subject to constraints defined in (\ref{eqn:conmiqcp_14}) and (\ref{eqn:qcqprev})-(\ref{eqn:qcqpmd}). The uniqueness and optimality of the solution to the QCQP formulation is given by Lemma \ref{lemma:phase2}.

\begin{lemma}[Uniqueness and Optimality of Phase 2]
	\label{lemma:phase2}
	$\forall h \in \mathbb Z_{> 0}$, given $S^h$, $R^h$ is uniquely determined in (\ref{eqn:step2}) via Phase 2 and it is optimal.  
\end{lemma} 
\begin{proof}
	$\forall h \in \mathbb Z_{> 0}$, given $S^h$, $R^h$ is derived by solving the QCQP formulation detailed in Section \ref{sub:phase2}. $\forall i \in \{0\} \cup Z_2$, $A_i$ is strictly convex, therefore $R_f$ is strictly convex as well. For each time cost term in (\ref{eqn:qcqpobjfun}), it is either an $L_2$ norm between two routing points ($t_i$) or an $L_2$ norm between one routing point and one requested location which is constant ($w_k^p$ or $w_k^d$), therefore strictly convex as well. Thus a nonnegative weighted sum of the time cost terms, i.e., $\mathcal C$, is strictly convex, as a result the QCQP formulation is strictly convex, thus $R^h$ is optimal and uniquely determined.
\end{proof}

The dimensionality of the decision variables and constraints of the QCQP formulation for Phase 2 of the AltMin algorithm are summarized in Table \ref{tab:dimensionqcqp}, where the number of variables and constraints are both linear with respect to the size of the problem, instead of quadratic dependency, due to the fact that the flow variables $x_{ij}$ are no longer in need as the sequence is fixed. Moreover, all decision variables are continuous, which further makes the QCQP formulation easy and efficient to solve.
\begin{table}[]
	\centering
	\caption{Dimensionality of the decision variables and constraints of the QCQP formulation for Phase 2.}
	\label{tab:dimensionqcqp}
	\begin{tabular}{ccr}
		\toprule
		\multicolumn{1}{c}{}        & \multicolumn{1}{c}{Category} & \multicolumn{1}{c}{Dimensionality} \\ \midrule
		\multirow{3}{*}{Variables}   & Binary                        & -                             \\
		& Continuous                    & $5N+6n+2$                        \\ \cmidrule{2-3} 
		& Sum                           & $5N+6n+2$                        \\ \midrule
		\multirow{4}{*}{Constraints}    & Quadratic                  & $2n$                                \\
		& SOC                           & $N+2n$                            \\
		& Linear                        & $2N+5n+2$                         \\ \cmidrule{2-3} 
		& Sum                           & $3N+9n+2$                          \\ 
		\bottomrule
	\end{tabular}
\end{table}

\subsection{Convergence Analysis and Stopping Criterion }
	\label{subsec:stop}

In this section, the convergence analysis of the AltMin algorithm is stated in Lemma \ref{lemma:convergence}. The stopping criterion of the mega iterations is discussed thereafter.

\begin{lemma}[Convergence of the AltMin Algorithm]
\label{lemma:convergence}
The AltMin algorithm either converges or recurs between a finite subset of the solution space.
\end{lemma} 
\begin{proof}
$S_f$ is a discrete set with finite cardinality, thus $\exists h_0 \in \mathbb Z_{> 0}$, such that $S^{h_0}=S^{h_0'}$, for some $h_0' \in \mathbb Z_{\ge 0}$ and $h_0'<h_0$, and $\forall h \in Z_{\ge 0}, h<h_0$, $S^h$ are distinct. According to Lemma \ref{lemma:phase1}, $\forall h \in Z_{> 0}$, given $R^{h-1}$, $S^h$ is uniquely determined via Phase 1. Similarly, according to Lemma \ref{lemma:phase2}, given $S^h$, $R^h$ is uniquely determined via Phase 2. Therefore $\forall h \in Z_{\ge 0}, (S^{h_0+h}, R^{h_0+h})=(S^{h_0'+h}, R^{h_0'+h})$, that being said, no new solution will be produced after mega iteration $h_0$. Moreover, if $h_0-h_0'=1$, $(S^h, R^h)$ converges to $(S^{h_0'}, R^{h_0'})$; otherwise, $(S^h, R^h)$ recurs between $(S^{h_0'}, R^{h_0'})$ and $(S^{h_0-1}, R^{h_0-1})$.
\end{proof}

It follows that in our implementation, once such an $h_0$ is encountered, the algorithm should stop as the subset of the solution space that can be reached by the AltMin algorithm has all been explored. In addition, an integer upper bound $h_{\text{max}}$ is also set on the number of mega iterations in order to limit the overall computational time. As a result, the actual number of mega iterations that the AltMin algorithm will perform is $\bar{h} = \text{min} \{h_0, h_{\text{max}} \}$. Hence the output of the algorithm is $\mathcal C^\ast=\underset{h \in \{1, 2, \dots, \bar{h}\}}{\text{min}} \mathcal \mathcal C_h$, denote $h^\ast=\underset{h \in \{1, 2, \dots, \bar{h}\}}{\text{argmin}} C_h$ and $S^\ast=S^{h^\ast}, R^\ast=R^{h^\ast}$. 

The overall pseudocode of the AltMin algorithm is summarized in Algorithm \ref{alg:AltMin}.

\begin{algorithm}[t]
	\SetAlgoNoLine
	\KwIn{depot $O$, ride requests $\{P_k, D_k, r_k^p, r_k^d, T_k^r\}, \forall k \in Z_1$, clustering pattern $\{C_1, C_2, \dots, C_N\}$, hyper parameters $\gamma_1, \gamma_2, \alpha_1, \alpha_2, \alpha_3^p, \alpha_3^d, c, MPS^p, MPS^d$, and mega iterations upper bound $h_{\text{max}}$}
	\BlankLine
	\textbf{Initialization:} $\mathcal C^\ast=\infty, S^\ast=\emptyset, R^\ast=\emptyset, R_i^0 \longleftarrow \text{center of mass of }A_i, \forall i \in \{0\} \cup Z_2$ 
	\BlankLine
	\For {$h = 1 : h_{\text{max}}$} {
		$S^h \longleftarrow \underset{S\in S_f}{\text{argmin}} \, \mathcal C(S, R^{h-1})$\\
		\If {$\exists h'<h, h' \in \mathbb Z_{> 0}, \text{s.t.} \, S^h=S^{h'}$} {\textbf{break}}\
		$R^h \longleftarrow \underset{R\in R_f}{\text{argmin}} \, \mathcal C(S^h, R)$\\
		$\mathcal C^h \longleftarrow \underset{R\in R_f}{\text{min}} \, \mathcal C(S^h, R)$\\
		\If {$\mathcal C_h<\mathcal C^\ast$} {
			$\mathcal C^\ast \longleftarrow \mathcal C_h$\\
			$S^\ast \longleftarrow S_h$\\
			$R^\ast \longleftarrow R_h$ 
		}
	}
	\BlankLine
	\KwOut{$\mathcal C^\ast$, $S^\ast$, $R^\ast$}
	\caption{The AltMin Algorithm of Static Routing with Space Windows}
	\label{alg:AltMin}
\end{algorithm}

\section{Computational Experiments}
\label{sec:experiments}

The AltMin algorithm described in Section \ref{sec:altmin} and the MIQCP formulation described in Section \ref{sec:miqcp} are now numerically evaluated using real operational data. The data was collected from the shuttle operation conducted in October 2016 by Chariot \cite{fordchariot, kampe} in San Francisco. The current operation therein consists of receiving passengers' requests regarding pickup, drop-off locations and start times of the trips, based on which a static route is designed. In order to evaluate our proposed AltMin algorithm, these pickup and drop-off locations, i.e., longitude and latitude, are utilized, thereby providing a realistic demand pattern for the evaluation. 

A total number of 26 instances with various scenarios, including size of the problem $(n, N)$, clustering patterns, passenger preferences, priority guarantees, and system objectives, have been simulated. For each of the instances, both the Altmin algorithm and the MIQCP based algorithm are implemented in MATLAB R2016b, where QCQP and MIQCP are implemented in Gurobi 7.5.1 with a MATLAB interface. A 3.20 GHz Intel(R) Xeon(R) CPU E5-1660 v4 desktop with 128GB of memory is used for all computational experiments. The fixed parameters are chosen such that $v_s=30 \, \text{mph}, v_p=3.1 \, \text{mph}, t_0=1 \, \text{min}$ and $a=2.25 \, \text{mph/s}$, which reasonably reflect the actual traffic conditions in urban areas.

The weights in (\ref{eqn:objfundef}) are chosen as $\gamma_1=1$, and $\alpha_1=2, \alpha_2=1$, with the assumption that a large waiting time is less preferred than a large riding time by the passengers. The shuttle capacity is set as $c=6$. A computational upper bound of one hour is set for all experiments. $h_{\text{max}}$ in the AltMin algorithm is set to 50 as an early stopping condition. Other problem settings which are subject to change across different instances are specified in Table \ref{tab:settings}. Each of the 26 instances is designed to represent different application scenarios. For instance, all scenarios where $\gamma_2=0$ correspond to the cases where the objective of the routing task is only to minimize the total travel time of the shuttle. Scenarios where $\text{MPS}^p={MPS}^d=2$ correspond to the cases where more emphasis is placed on the queue order and the dynamic shuttle services are more in a first come first serve fashion. 

% Please add the following required packages to your document preamble:
% \usepackage{multirow}
\begin{table}[]
	\centering
	\caption{Settings of the computational experiments.}
	\label{tab:settings}
	\begin{tabular}{ccrrcccc}
		\toprule
		\multirow{2}{*}{Index} & \multirow{2}{*}{Instance} & \multicolumn{2}{c}{Problem Size} & \multicolumn{4}{c}{Hyper Parameters}                                          \\
		\cmidrule(lr){3-4}
		\cmidrule(lr){5-8}
		&                           & $n$                & $N$                & $\gamma_2$ & $\alpha_3^p/\alpha_3^d$ & $r^p/r^d$ & $\text{MPS}^p/\text{MPS}^d$ \\ \midrule
		1                      & p06-c12-1                 & 6                  & 12                 & 0          & -                       & 0.3       & 6                          \\
		2                      & p06-c06-1                 & 6                  & 6                  & 0          & -                       & 0.3       & 6                          \\
		3                      & p06-c12-2                 & 6                  & 12                 & 1          & 0.1                     & 0.3       & 6                          \\
		4                      & p06-c07-1                 & 6                  & 7                  & 1          & 0.1                     & 0.3       & 6                          \\
		5                      & p06-c12-3                 & 6                  & 12                 & 1          & 0.1                     & 0.3       & 2                          \\
		6                      & p06-c07-2                 & 6                  & 7                  & 1          & 0.1                     & 0.3       & 2                          \\
		7                      & p06-c12-4                 & 6                  & 12                 & 0          & -                       & 0.15      & 6                          \\
		8                      & p06-c08-1                 & 6                  & 8                  & 0          & -                       & 0.15      & 6                          \\
		9                      & p08-c16-1                 & 8                  & 16                 & 0          & -                       & 0.3       & 6                          \\
		10                     & p08-c08-1                 & 8                  & 8                  & 0          & -                       & 0.3       & 6                          \\
		11                     & p08-c16-2                 & 8                  & 16                 & 1          & 0.1                     & 0.3       & 6                          \\
		12                     & p08-c09-1                 & 8                  & 9                  & 1          & 0.1                     & 0.3       & 6                          \\
		13                     & p08-c16-3                 & 8                  & 16                 & 1          & 0.1                     & 0.3       & 2                          \\
		14                     & p08-c11-1                 & 8                  & 11                 & 1          & 0.1                     & 0.3       & 2                          \\
		15                     & p08-c16-4                 & 8                  & 16                 & 0          & -                       & 0.15      & 6                          \\
		16                     & p08-c09-2                 & 8                  & 9                  & 0          & -                       & 0.15      & 6                          \\
		17                     & p10-c20-1                 & 10                 & 20                 & 1          & 0                       & 0.3       & 6                          \\
		18                     & p10-c20-2                 & 10                 & 20                 & 0          & -                       & 0.3       & 6                          \\
		19                     & p10-c12-1                 & 10                 & 12                 & 0          & -                       & 0.3       & 6                          \\
		20                     & p10-c20-3                 & 10                 & 20                 & 1          & 0.1                     & 0.3       & 6                          \\
		21                     & p10-c11-1                 & 10                 & 11                 & 1          & 0.1                     & 0.3       & 6                          \\
		22                     & p10-c20-4                 & 10                 & 20                 & 1          & 1                       & 0.3       & 6                          \\
		23                     & p10-c20-5                 & 10                 & 20                 & 1          & 0.1                     & 0.3       & 2                          \\
		24                     & p10-c13-1                 & 10                 & 13                 & 1          & 0.1                     & 0.3       & 2                          \\
		25                     & p10-c20-6                 & 10                 & 20                 & 0          & -                       & 0.15      & 6                          \\
		26                     & p10-c16-1                 & 10                 & 16                 & 0          & -                       & 0.15      & 6                          \\
		\bottomrule
	\end{tabular}
\end{table}

The computational results obtained for the 26 instances are summarized in Table \ref{tab:results}. Three columns are included to describe the performances of AltMin, and five for MIQCP. The three columns of AltMin correspond to $\bar{h}$, the actual number of overall iterations of Phase 1 and 2 when convergence is reached, $\mathcal C_{\text{AltMin}}^\ast$ (unit [h]) the cost of the system obtained, and the CPU time (unit [s]). And the five columns of MIQCP correspond to \textit{Vars} and \textit{Cons}, the total numbers of decision variables and constraints in the formulation, respectively, \textit{Bound} (unit [h]) the best lower bound of the cost obtained in the branch-and-cut implementation of Gurobi, $\mathcal C_{\text{MIQCP}}^\ast$ cost of the system obtained, and the CPU time. In the latter, it follows that if the best lower bound equals $\mathcal C_{\text{MIQCP}}^\ast$, then optimality is achieved. We include an additional column corresponding to \textit{Optimality Gap} to benchmark the optimality of AltMin with respect to MIQCP, defined as
\begin{equation}
	\label{eqn:optimality}
	\text{optimality gap}=\frac{\mathcal C_{\text{AltMin}}^\ast-\mathcal C_{\text{MIQCP}}^\ast}{\mathcal C_{\text{MIQCP}}^\ast}, \quad \text{if} \,\, \mathcal C_{\text{MIQCP}}^\ast \leq \mathcal C_{\text{AltMin}}^\ast
\end{equation}
If $\mathcal C_{\text{MIQCP}}^\ast > \mathcal C_{\text{AltMin}}^\ast$, it is clear that AltMin outperforms MIQCP in terms of the cost optimality. A diamond symbol $\diamond$ is included at the beginning of a row where Altmin outperforms MIQCP in terms of both optimality and computational efficiency. An asterisk symbol $\ast$ is included at the beginning of the row where the optimality gap of the AltMin algorithm is zero.

% Please add the following required packages to your document preamble:
% \usepackage{multirow}
\begin{table}[]
	\centering
	\caption{Results of the computational experiments.}
	\label{tab:results}
	\begin{tabular}{rccrrrcccc}
		\toprule
		\multicolumn{1}{c}{\multirow{2}{*}{Instance}} & \multicolumn{3}{c}{AltMin}                                                                                                                                          & \multicolumn{5}{c}{MIQCP}                                                                                                                                                                                                        & \multirow{2}{*}{\begin{tabular}[c]{@{}c@{}}Optimality\\ Gap\end{tabular}} \\
		\cmidrule(llr){2-4}
		\cmidrule(lr){5-9}
		& \begin{tabular}[c]{@{}c@{}}$\bar{h}$\end{tabular} & \multicolumn{1}{c}{$\mathcal C_{\text{AltMin}}^\ast$} & \multicolumn{1}{c}{\begin{tabular}[c]{@{}c@{}}CPU\end{tabular}} & \multicolumn{1}{c}{Vars} & \multicolumn{1}{c}{Cons} & \multicolumn{1}{c}{\begin{tabular}[c]{@{}c@{}}Bound\end{tabular}} & \multicolumn{1}{c}{$\mathcal C_{\text{MIQCP}}^\ast$} & \multicolumn{1}{c}{\begin{tabular}[c]{@{}c@{}}CPU\end{tabular}} &                                                                           \\ \midrule
		$\diamond$p06-c12-1                 & 2                                                               & 0.469           & 2.409                                                                           & 706                      & 1,187                    & 0.356                                                                    & 0.469           & 3,600                                                                            & -                                                                         \\
		$\ast$p06-c06-1                 & 2                                                               & 0.361                    & 0.740                                                                           & 232                      & 389                      & 0.361                                                                   & 0.361                     & 0.283                                                                           & 0                                                                         \\
		p06-c12-2                 & 2                                                               & 3.497                     & 2.383                                                                           & 706                      & 1,187                    & 2.024                                                                    & 3.482                     & 3,500                                                                            & 0.41\%                                                                    \\
		p06-c07-1                 & 2                                                               & 2.914                     & 1.009                                                                           & 291                      & 487                      & 2.886                                                                    & 2.886                     & 0.343                                                                           & 0.96\%                                                                    \\
		$\ast$p06-c12-3                 & 2                                                               & 3.482                    & 1.856                                                                           & 706                      & 1,187                    & 3.482                                                                   & 3.482                     & 716.2                                                                           & 0                                                                         \\
		$\ast$p06-c07-2                 & 2                                                               & 3.262                    & 0.803                                                                           & 291                      & 487                      & 3.262                                                                   & 3.262                     & 0.261                                                                           & 0                                                                         \\
		$\ast$p06-c12-4                 & 2                                                               & 0.504                    & 2.355                                                                           & 706                      & 1,187                    & 0.504                                                                   & 0.504                     & 2,368                                                                            & 0                                                                         \\
		$\ast$p06-c08-1                 & 2                                                               & 0.430                    & 1.028                                                                           & 358                      & 599                      & 0.430                                                                   & 0.430                     & 1.053                                                                           & 0                                                                         \\
		$\diamond$p08-c16-1                 & 3                                                               & 0.583           & 52.01                                                                           & 1,194                    & 2,027                    & 0.047                                                                    & 0.626           & 3,600                                                                            & -                                                                         \\
		$\ast$p08-c08-1                 & 2                                                               & 0.548                    & 4.645                                                                           & 370                      & 627                      & 0.548                                                                   & 0.548                     & 1.889                                                                           & 0                                                                         \\
		$\diamond$p08-c16-2                 & 2                                                               & 5.439           & 34.13                                                                           & 1,194                    & 2,027                    & 2.052                                                                    & 5.446           & 3,600                                                                            & -                                                                         \\
		p08-c09-1                 & 2                                                               & 4.926                     & 6.228                                                                           & 445                      & 753                      & 4.538                                                                    & 4.438                     & 5.188                                                                           & 8.55\%                                                                    \\
		$\diamond$p08-c16-3                 & 2                                                               & 6.586           & 20.26                                                                           & 1,194                    & 2,027                    & 2.259                                                                    & 6.603           & 3,600                                                                            & -                                                                         \\
		$\ast$p08-c11-1                 & 2                                                               & 5.995                    & 6.689                                                                           & 619                      & 1,047                    & 5.995                                                                   & 5.995                     & 8.172                                                                           & 0                                                                         \\
		$\diamond$p08-c16-4                 & 2                                                               & 0.629           & 33.76                                                                           & 1,194                    & 2,027                    & 0.052                                                                    & 0.670           & 3,600                                                                            & -                                                                         \\
		$\ast$p08-c09-2                 & 2                                                               & 0.521                    & 6.560                                                                           & 445                      & 753                      & 0.521                                                                   & 0.521                     & 6.146                                                                           & 0                                                                         \\
		$\diamond$p10-c20-1                 & 2                                                               & 8.482           & 560.8                                                                           & 1,810                    & 3,091                    & 2.454                                                                    & 8.484           & 3,600                                                                            & -                                                                         \\
		$\diamond$p10-c20-2                 & 2                                                               & 0.767           & 579.9                                                                           & 1,810                    & 3,091                    & 0.064                                                                    & 1.090           & 3,600                                                                            & -                                                                         \\
		p10-c12-1                 & 2                                                               & 0.710                     & 120.8                                                                           & 730                      & 1,243                    & 0.699                                                                    & 0.699                     & 238.6                                                                           & 1.54\%                                                                    \\
		$\diamond$p10-c20-3                 & 2                                                               & 8.658           & 551.4                                                                           & 1,810                    & 3,091                    & 2.567                                                                    & 9.227           & 3,600                                                                            & -                                                                         \\
		$\ast$p10-c11-1                 & 2                                                               & 8.344                    & 67.66                                                                           & 631                      & 1,075                    & 8.344                                                                   & 8.344                     & 159.8                                                                           & 0                                                                         \\
		$\diamond$p10-c20-4                 & 2                                                               & 9.373           & 554.8                                                                           & 1,810                    & 3,091                    & 2.628                                                                    & 9.491           & 3,600                                                                            & -                                                                         \\
		$\diamond$p10-c20-5                 & 2                                                               & 9.205           & 273.9                                                                           & 1,810                    & 3,091                    & 2.502                                                                    & 9.278           & 3,600                                                                            & -                                                                         \\
		$\ast$p10-c13-1                 & 3                                                               & 7.895                    & 87.83                                                                           & 837                      & 1,425                    & 7.895                                                                   & 7.895                     & 65.74                                                                           & 0                                                                         \\
		$\diamond$p10-c20-6                 & 3                                                               & 0.817           & 752.2                                                                           & 1,810                    & 3,091                    & 0.052                                                                    & 1.072           & 3,600                                                                            & -                                                                         \\
		$\diamond$p10-c16-1                 & 3                                                               & 0.769           & 401.8                                                                           & 1,206                    & 2,055                    & 0.066                                                                    & 0.848           & 3,600                                                                            & -                                                                        \\
		\bottomrule
	\end{tabular}
\end{table}

It is clear from Table \ref{tab:results} that AltMin outperforms MIQCP in most of the instances, regarding both optimality and computational efficiency. More detailed observations are made below. 
\begin{itemize}
	\item AltMin is significantly faster than MIQCP. The average CPU time of all instances for AltMin is 158.8 seconds, which is more than one order smaller than that of 1,937 seconds for MIQCP. It should be noted that a one-hour computational cap is placed.
	\item Among the 26 instances, AltMin converges to optimality in 10 instances, and with better computational efficiency than MIQCP; average CPU time for AltMin is 19.81 seconds versus 290.1 seconds for MIQCP. These instances are the ones marked by an asterisk $\ast$ in Table \ref{tab:results}.
	\item In 12 out of the 26 instances marked by a diamond $\diamond$, AltMin outperforms MIQCP in terms of both optimality and computational efficiency.  
	\item In only 4 instances, MIQCP yields a better cost than AltMin, however the optimality gap is fairly small. This indicates that in some scenarios, even though AltMin converges to a suboptimal solution, it is fairly close to the optimal. The computational time remains much smaller in these cases as well when compared to MIQCP.
	\item In all of the computational experiments, AltMin displays a fast convergence, with only 2 or 3 overall iterations of Phase 1 and 2.
\end{itemize}

For different instances with the $(n, N)$, a smaller $r^p/r^d$ or $\text{MPS}^p/\text{MPS}^d$ yields faster computations and better optimality, for both AltMin and MIQCP. This is due to the fact that though the problem sizes are the same, smaller $r^p/r^d$ or $\text{MPS}^p/\text{MPS}^d$ values shrink the size the solution space (though the asymptotic running time still the same), which facilitate faster and more optimal solutions. Similarly, a smaller $c$ value shares same impacts, which is not demonstrated here.  

The computational efficiency of the AltMin algorithm can be improved even further by i) utilizing more powerful computational resources, and ii) exploiting heuristics for both Phase 1 (e.g., $k$-opt, simulated annealing, ant colony optimization) and Phase 2 (e.g., greedy algorithms).

Fig. \ref{fig:static} demonstrates the overall routes of 4 instances derived using AltMin. It should be noted that the set of requests, i.e., requested pickup and drop-off locations, are the same for the 4 instances. The reason why the actual routes realized are different is due to the differences in the corresponding problem settings shown in Table \ref{tab:settings}. 
\begin{figure}[h!]
	\centering
	\includegraphics[width=1\textwidth]{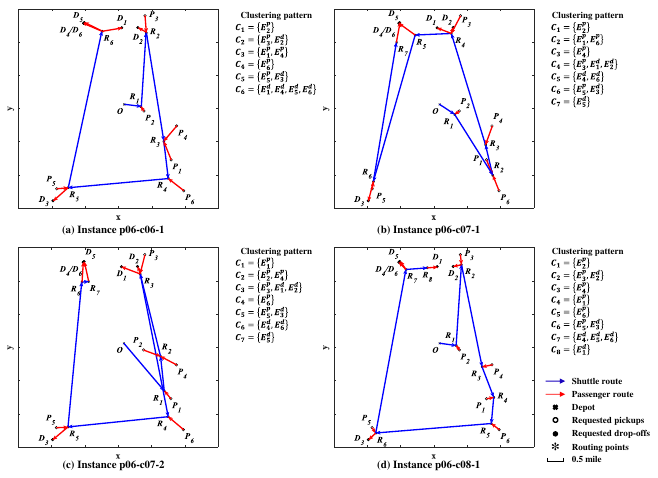}
	\caption{Overall routes of 4 instances of the same set of requests with different problem settings.}
	\label{fig:static}
\end{figure}

\section{Extension to Dynamic Routing}
\label{sec:dynamic}

In this section, the algorithm for dynamic routing is discussed using static routing as a subcomponent. Dynamic routing with space window preserves the advantages of its static counterpart, and is more flexible in terms of how to accommodate new requests and achieve better system level performance. In the dynamic routing scheme, the shuttle server searches for new requests periodically and makes suitable updates (using a modified version of the proposed static routing algorithm, detailed below) on its current route to take care of the new requests immediately as long as all constraints can be satisfied and a reduction on the cost of the system can be achieved.

The main difference between dynamic and static routing is that any update on the route can affect those passengers who have either been assigned pickups but have not been picked up yet or already been on the shuttle, such that their actual pickups/drop-offs might be delayed. We denote these passengers as \textit{existing passengers}. However, as long as the delays are below certain upper bounds specified in the initial ride offers provide by the shuttle server, the delays are considered acceptable.

Overall, one can view dynamic routing as a generalized version of its static counterpart, with a periodical update on the static route in response to real time requests from new passengers. Below, we summarize how we modify the AltMin algorithm described in Section \ref{sec:altmin} and the MIQCP formulation described in Section \ref{sec:miqcp} to accommodate this update. It should be noted that the same assumption is made that the updated clustering pattern is explicitly given beforehand.
\begin{enumerate}
	\item Modified dummy initial cluster $C_0$. \\
	In static routing, the dummy initial cluster is set as $C_0=\emptyset$. In dynamic routing, $C_0$ needs to be modified in order to accommodate the passengers who are already on board when updating the route. In particular, $C_0$ consists of the pickup events of the passengers who are on board (if any), with the corresponding $A_0$ as the shuttle's exact location when the update occurs. %The drop-off events of the passengers on board, the pickup and drop-off events of the passengers who have requested rides however have not yet been taken care of, contribute to the updated clustering pattern.
	\item Frozen routing points of existing but unexecuted pickup events. \\
	The clustering pattern, sequence of the clusters to be visited and routing points are all subject to change when any update occurs. An additional constraint in dynamic routing should be included that the routing points corresponding to the pickup events of the existing passengers who are yet to be picked up have to be frozen, and are no longer decision variables. Otherwise, additional walking times may be imposed on these passengers. Though the overall walking distance might still be within the specified maximum walking distance $r_k^p$, the \textit{double walk} may cause huge negative externalities to the riding experiences of these passengers, which is not desirable. This can be easily accomplished by setting $S_W(P_k, r_k^p)=S_W(R_{c\ell_k^p},0), \forall k$ that corresponds to an existing passenger who has not yet been picked up. It should be noted that the negotiated drop-off locations of the existing passengers are not necessarily frozen as the walks to the final destinations have not yet occurred, thus no such double walk will happen because of the updates on the current route (Fig. \ref{fig:double}).
	\begin{figure}[h!]
		\centering
		\includegraphics[width=0.95\textwidth]{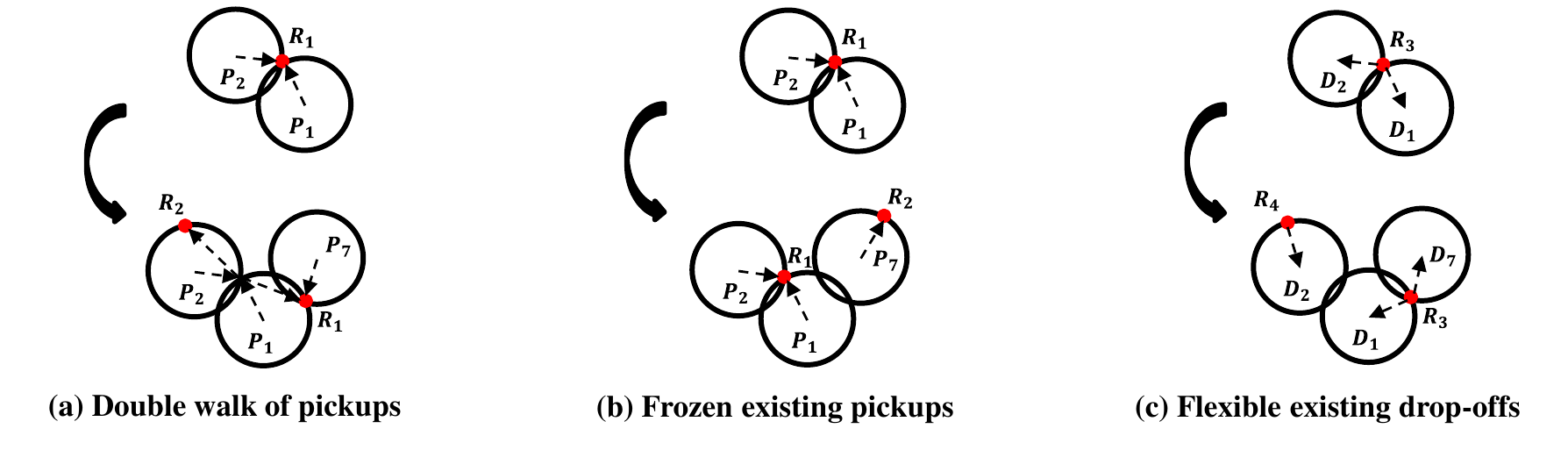}
		\caption{Demonstrations of double walks. (a): Double walks caused by updating $R_1$ which serves as the negotiated pickup locations for passenger 1 and 2, who have not been picked up yet. (b): Double walks are prevented by freezing $R_1$. (c): $R_3$ is flexible as it serves as the negotiated drop-off location for passenger 1 and 2, whose updates will not result in any double walk.}
		\label{fig:double}
	\end{figure}
	\item Consistent references of time frame and queues. \\
	Static routing only takes care of one batch of requests, while dynamic routing conducts static routing periodically, where consecutive batches of requests are accommodated. The references of time frame and the three queues associated with the passengers, i.e., $\{k\}$, $\{q_k\}$ and $\{d_k\}$ should be consistent throughout different batches. The consistency can be guaranteed by simply using the absolute references for both the time frame and the three queues. More specifically, for any passenger $k, \forall k \in \mathbb Z_{> 0}$, $k$, $p_k$ and $d_k$ are his/her absolute sequences in the entire request, pickup and drop-off queue, respectively, instead of the relative sequences in the current batch. And then the MPS constraints are able to guarantee that no passenger is likely to experience any indefinite deferment due to the updates on the routes. Similarly, the waiting and riding time of any passenger should be derived via collapsing between the absolute time stamps when he/she gets picked up and requests the shuttle ride, and between those when he/she gets dropped off and picked up, respectively.   
\end{enumerate}

Dynamic routing has been demonstrated to be advantageous over its static counterpart when the requests arrive frequently, more specifically, when the time intervals between different batches of requests are shorter than those it takes to serve all existing passengers \cite{annaswamy2018transactive}. However, static routing is easier to operate and understand, passengers are not subject to any delays on their pickups or drop-offs, and the drivers will not suffer from fatigue due to frequent updates on the shuttle routes (certainly this issue could be alleviated significantly if autonomous driving vehicles could come into play). Therefore, if the requests are not frequent compared to the duration of the existing trips, static routing suffices. However, in any scenario with a high travel demand, dynamic routing may be unavoidable and necessary.

Fig. \ref{fig:dynamic} demonstrates a numerical example of how the AltMin dynamic routing algorithm operates to serve a total number of 12 passengers whose requests are received in 2 batches, 6 requests per batch. The location data (i.e., longitude and latitude of requested pickup and drop-off locations) is based on the \textit{GoRide} dynamic shuttle service on Ford's Dearborn campus in August 2016 \cite{forddynamicshuttle}. The first batch is supposed to be finished in 24.3 minutes (derived via the AltMin static routing algorithm), while the second batch was received at $T^r=15$ minutes, when passenger 2 and 4 are on aboard and passenger 3 has not been picked up yet. Fig. \ref{fig:dynamic}(d) and (e) compare the overall routes to serve all 12 passengers using static and dynamic routing, respectively. It is clear that via dynamic routing, the shuttle makes suitable updates, e.g., $R_7$, and achieves a better system level performance.  

\begin{figure}[h!]
	\centering
	\includegraphics[width=1\textwidth]{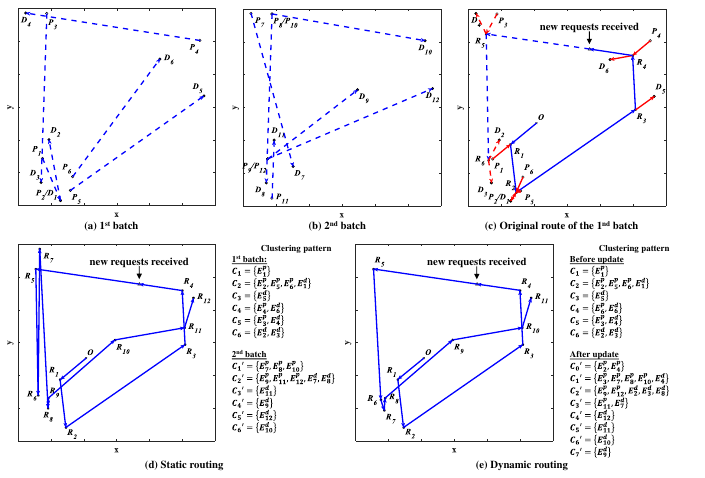}
	\caption{Demonstrations of dynamic routing via AltMin. (a), (b): The Origin-Destination (OD) pairs of the two batches of requests, respectively. (c): The original route of the first batch, dashed arrows correspond to the part which have not been carried out yet when the second batch is received; (d), (e): The overall routes to serve all 12 passengers, via static and dynamic routing, respectively. One tick corresponds to 0.5 mile.}
	\label{fig:dynamic}
\end{figure}

\section{Summary and Conclusions}
\label{sec:conclusions}

In this paper, we have proposed a novel dynamic routing framework for shared mobility services which utilizes multi-passenger shuttles, facilitated by the central concept of space window. In addition to the development of the framework, the contributions of the paper include a constrained optimization formulation that shows how the space window concept can be utilized to carry out dynamic routing of shuttles, and an AltMin algorithm composed of a modified BHK algorithm as Phase 1 and a QCQP formulation as Phase 2 to solve the constrained optimization problem. As shown in Section \ref{sec:experiments}, this AltMin algorithm is demonstrated to be an order of magnitude better in terms of computational efficiency and comparable in terms of optimality to a standard MIQCP based approach. Overall, the proposed dynamic routing framework is capable of reducing the underlying travel time costs, and enriches the scope of the emerging shared MoD services, providing an ideal combination of flexibility, convenience, and affordability for people to move around.

Future works will concentrate on the development of heuristics to solve Phase 1 of the AltMin algorithm in order to accommodate the scenarios where larger capacity shuttles are in use. An analytical and more comprehensive analysis on the optimality of the AltMin algorithm, as well as approaches such as simulated annealing so that the overall iterations of Phases 1 and 2 do not end in a local minimum, are of interest. The determination process of the clustering pattern will also be investigated systematically. The extension to a distributed optimization framework for a multi-shuttle scenario and that of dynamic pricing strategies \cite{annaswamy2018transactive} are our future directions as well.

\bibliographystyle{ACM-Reference-Format}

\bibliography{references}

\end{document}